\documentclass[twoside,11pt]{article}

\usepackage{blindtext}

\usepackage[abbrvbib, preprint]{jmlr2e}

\usepackage{amsmath}
\usepackage{amsfonts,mathtools}
\mathtoolsset{showonlyrefs}    
\usepackage{graphicx}
\usepackage{epstopdf}
\usepackage{algorithmic,enumitem}
\usepackage[utf8]{inputenc}
\usepackage[margin=1in]{geometry}
\usepackage{graphicx}
\usepackage{xcolor}
\usepackage{hyperref}
\usepackage{bm}
\usepackage[mathscr]{euscript}
\usepackage{subfig}
\usepackage{mathabx}
\usepackage{enumitem}

\newtheorem{mainthm}{Formal Theorem}

\newcommand{\Rb}{\mathbb{R}}

\newcommand{\Eb}{\mathbb{E}}
\newcommand{\Nb}{\mathbb{N}}

\newcommand{\Ec}{\mathcal{E}}

\newcommand{\Kc}{\mathcal{K}}
\newcommand{\Lc}{\mathcal{L}}
\newcommand{\Nc}{\mathcal{N}}
\newcommand{\Pc}{\mathcal{P}}

\newcommand{\Vc}{\mathcal{V}}

\newcommand{\rmc}{\mathrm{c}}
\newcommand{\rmd}{\mathrm{d}}
\newcommand{\rmb}{\mathrm{b}}

\newcommand{\rmo}{\mathrm{o}}

\newcommand{\rmE}{\mathrm{E}}

\newcommand{\AC}{\mathrm{AC}}

\newcommand{\eps}{\varepsilon}
\newcommand{\sfr}{\mathsf{r}}

\DeclareMathOperator*{\argmin}{\mathrm{argmin}}

\usepackage{lastpage}
\jmlrheading{??}{2023}{1-\pageref{LastPage}}{10/23; Revised ?/??}{?/??}{21-0000}{Shi Chen, Qin Li, Oliver Tse, and Stephen J. Wright}

\ShortHeadings{Acceleration for optimizing probability measures}{Chen, Li, Tse, and Wright}
\firstpageno{1}

\begin{document}

\title{Accelerating optimization over the space of probability measures}

\author{\name Shi Chen \email schen636@wisc.edu \\
        \addr Department of Mathematics\\
        University of Wisconsin-Madison\\
        Madison, WI 53706, USA
        \AND 
        \name Qin Li \email qinli@math.wisc.edu\\
        \addr Department of Mathematics\\
        University of Wisconsin-Madison\\
        Madison, WI 53706, USA
        \AND 
        \name Oliver Tse \email o.t.c.tse@tue.nl\\
        \addr Department of Mathematics and Computer Science and Eindhoven Hendrik Casimir Institute\\ 
        Eindhoven University of Technology\\
        5600 MB Eindhoven, The Netherlands 
        \AND Stephen J. Wright \email swright@cs.wisc.edu \\
        \addr Department of Computer Sciences\\
        University of Wisconsin-Madison\\
        Madison, WI 53706, USA}

\editor{My editor}

\maketitle

\begin{abstract}
The acceleration of gradient-based optimization methods is a subject of significant practical and theoretical importance, particularly within machine learning applications. 
While much attention has been directed towards optimizing within Euclidean space, the need to optimize over spaces of probability measures in machine learning motivates exploration of accelerated gradient methods in this context too.
To this end, we introduce a Hamiltonian-flow approach analogous to momentum-based approaches in Euclidean space. We demonstrate that, in the continuous-time setting, algorithms based on this approach can achieve convergence rates of arbitrarily high order. We complement our findings with numerical examples.
\end{abstract}


\section{Introduction}
The search for a probability measure that minimizes an objective functional plays a significant role across many machine learning problems, encompassing areas such as generative modeling~\citep{KiWe:2013auto,GoPoMiXuWaOzCoBe:2014generative,SoWeMaGa:2015deep,HoJaAb:2020denoising,SoSoKiKuErPo:2020score,BeDuVi:2000neural}, Bayesian inference~\citep{JoGhJaSa:1999introduction,WaJo:2008graphical,HoBlWaPa:2013stochastic,BlKuMc:2017variational,ReMo:2015variational,LaChBaBoRi:2022variational,GeDo:2023langevin} and reinforcement learning~\citep{ZiMaBaDe:2008maximum,To:2009robot,PeMuAl:2010relative,Le:2018reinforcement}. 
These problems are stated as
\begin{equation}\label{eqn:optimization_E}
\rho_\ast=\argmin_{\rho \in \Pc(\Omega)} \, E[\rho],\quad\text{and}\quad E_\ast \coloneqq E[\rho_\ast],
\end{equation}
where $\Pc(\Omega)$ is the collection of all probability measures supported on $\Omega\subset\mathbb{R}^d$ and $E:\Pc(\Omega)\to \Rb$ maps probability measures to $\Rb$. 
Throughout the paper, we use $[\cdot]$ to denote the dependence of a functional on a function/probability measure. With a slight abuse of notation, we do not distinguish probability measures from their Lebesgue densities. 

In this article, we focus on the continuous-time setting for solving~\eqref{eqn:optimization_E}.
Analogous to the \emph{gradient flow}
in Euclidean space, a \emph{gradient flow} on the space of probability measures can be used to find the minimizer. 
We define a gradient flow in $\Pc(\Omega)$ by making $\rho$ depend on a ``time'' variable $t \ge 0$ (notation: $\rho_t$) and writing
\begin{equation}\label{eqn:gradient_flow}
\partial_t\rho_t = -\nabla_{D}E[\rho_t]\,, \tag{{\sf GF}}
\end{equation}
where $\nabla_{D}$ is an appropriately chosen metric on the space of probability measures. A commonly employed metric in the literature is the Wasserstein metric~\citep{Ot:2001geometry}. 
The gradient flow determined by the Wasserstein gradient exhibits the same convergence rate as that of gradient flow in Euclidean spaces; see Table~\ref{tbl:opt}.

Given the numerous strategies developed to accelerate first-order (gradient-based) optimization methods in Euclidean space, it is natural to seek counterparts of these methods to optimization over the space of probability measures. Specifically, we are curious as to whether momentum-based acceleration methods on Euclidean space can be adapted to optimization problems of the form~\eqref{eqn:optimization_E}. Our focus lies on the convergence properties in continuous time, prompting the following questions:
\begin{itemize}[itemsep=0pt]
    \item Is there a \emph{``heavy-ball"} method applicable to \eqref{eqn:optimization_E}, and is it provably faster than gradient flow \eqref{eqn:gradient_flow}, akin to its superior performance for convex objective functions in Euclidean space? 
    \item Is there a \emph{``Nesterov acceleration"} method applicable to \eqref{eqn:optimization_E}, does it exhibit provable speedup over gradient flow?
    \item What is the optimal convergence rate achievable by a first-order algorithm for \eqref{eqn:optimization_E} when employing a momentum strategy?
\end{itemize}

In addressing these questions, we design algorithms for minimizing over the space probability measures, resembling heavy-ball methods, variational acceleration methods (with Nesterov acceleration being one instance), and, more generally, Hamiltonian flows. 
These new algorithms demonstrate provable acceleration over regular gradient flow, mirroring their counterparts in Euclidean space. 
By careful design of the Hamiltonian, we achieve convergence of $e^{-\beta_t}$ for almost all choices of $\beta_t$, including $\beta_t = t$ and $\log t^p$ for any power $p >0$; see Table~\ref{tbl:opt}.
\begin{table}
  \centering
  \begin{tabular}{ l | cc cc }
    \hline \hline
      Property & \multicolumn{2}{c}{Convex} & \multicolumn{2}{c}{$m$-Strongly Convex} \\
      \hline
      Objective & $f(x)$ & $E[\rho]$ & $f(x)$ & $E[\rho]$ \\
    \hline
    Gradient Flow  & \multicolumn{2}{c}{$O(1/t)$} & \multicolumn{2}{c}{$O(e^{-2m t})$}\\
    Heavy-ball Method & \multicolumn{2}{c}{$o(1/t)$} & \multicolumn{2}{c}{$O(e^{-\sqrt{m}t})$
    } \\
    Variational Acceleration & \multicolumn{2}{c}{$O(e^{-\beta_t})$} & \multicolumn{2}{c}{-}\\
    \hline\hline
  \end{tabular}
    \caption{Convergence rates of three momentum-based methods. All three methods share the same convergence rate to optimize a finite-dimensional function $f(x)$ and a functional $E[\rho]$. 
    }
  \label{tbl:opt}
\end{table}

These developments hinge on two observations:
\begin{itemize}[itemsep=0pt]
    \item In Euclidean space, many momentum-based acceleration methods rely on a carefully crafted Hamiltonian term $h_t:\Rb^d\times\Rb^d\to\Rb$. It is shown, for various methods, that as long as particles follow trajectories defined by
    \begin{equation}\label{eqn:Hamiltonian_ODE1}
    \dot{x}=\nabla_v h_t(x,v)\,,\qquad\dot{v} = -\nabla_x h_t(x,v)\,,
    \end{equation}
    these particles can descend to the global minimum faster than the classical gradient flow. 
    To obtain the acceleration effect in the space of probability measures, we need to conjure up the ``Hamiltonian'' concept for these spaces.
    
    \item The second observation is that a probability measure can be approximated by its samples:
\[
\rho\approx\frac{1}{N}\sum_{i=1}^N\delta_{x_i} \in\Pc(\Rb^d)\,\qquad N\in\Nb\,.
\]
The evolution of $\rho$ in the space of probability measures can be fully translated to the motion of its representative samples $\{x_1,x_2,\dotsc,x_N\}$.
In the current context, methods such as heavy-ball and Nesterov prescribe the motion of particles by assigning dynamics to $(x_i,v_i)$. 
By deploying the relation between the motion of the particle sample and the evolution of its corresponding probability measure, we can lift these methods to define an evolution for $\rho$.
\end{itemize}

\bigskip

With these two observations in hand, and noting that each particle can be fully described by $(x_i,v_i)\in \Rb^d\times\Rb^d$, we define the empirical measure
\begin{equation}
\mu = \frac{1}{N} \sum_{i=1}^N \delta_{(x_i,v_i)} \in \Pc(\Rb^d\times\Rb^d)\,,\quad N\in\Nb\,,
\end{equation}
over this extended space---henceforth called the {\em phase space}.
Deploying the relationship between the motion of particles and the empirical measure $\mu$, and utilizing~\eqref{eqn:Hamiltonian_ODE1}, we arrive at the \emph{Hamiltonian flow} equation: 
\begin{equation}\label{eqn:hamiltonian_flow_intro}
\partial_t \mu_t 
+ \nabla_x\cdot \left(\mu_t\nabla_v\frac{\delta H_t}{\delta \mu}[\mu_t]\right) 
- \nabla_v\cdot \left(\mu_t\nabla_x\frac{\delta H_t}{\delta \mu}[\mu_t]\right) = 0\,,
\end{equation}
where $H_t$ is the counterpart of $h_t$ lifted to the space of probability measures defined by
\[
H_t[\mu] = \int_{\Rb^{2d}} h_t(x,v)\, \rmd \mu(x,v) \,,
\]
and with $\frac{\delta H_t}{\delta\mu}[\mu_t]$ being its variational (Fr\'echet) derivative evaluated at $\mu_t$\footnote{The Fr\'echet derivative is the counterpart of Euclidean derivatives in a function space. The Euclidean derivative of a function $\nabla f(x)$ measures the first-order differentiation of this function at a point $x$: $f(x+\Delta x)-f(x)\approx \langle\nabla f(x)\,,\Delta x\rangle$, with the bracket notation denoting the inner product defined on Euclidean space. Since $x\in\mathbb{R}^d$, we have $\nabla f(x)\in\mathbb{R}^d$.
Similarly, the Fr\'echet derivative quantifies the first-order differentiation of a functional over the change in a function:
\[
	H[\mu+\delta \mu] - H[\mu] \approx \left\langle \frac{\delta H}{\delta \mu} [\mu], \delta \mu \right \rangle \coloneqq \int \frac{\delta H}{\delta \mu} [\mu](z) \, \delta \mu(z)\,\rmd z\,,
\]
with the bracket notation denoting the inner product or duality pairing on the function space. In the following, we use $\frac{\delta}{\delta \rho}$ and $\frac{\delta}{\delta \mu}$ to denote the Fr\'echet derivatives of functionals that take in probability measures over $\Rb^d$ and $\Rb^{d}\times\Rb^d$, respectively.}.

The formulation~\eqref{eqn:hamiltonian_flow_intro} serves as the foundation of our algorithm design. In particular, it is devised under the assumption that $H_t$ and $h_t$ are connected via a linear form, yet it remains valid for a general form of $H_t$. More precisely, it holds when
\begin{equation}
    H_t[\mu] = K_t[\mu] + P_t[\mu]\,,
\end{equation}
with $K_t$ and $P_t$ representing the kinetic and potential energy functionals, respectively. The core of our strategy lies in carefully crafting the form of these functionals and connecting them to $E[\rho]$ to achieve acceleration.

As in the Euclidean setting, the design of Hamiltonian induces a variety of convergence behaviors. 
By tailoring $H_t$ to be consistent with the Hamiltonian term $h_t$ crafted for the heavy-ball and Nesterov acceleration in Euclidean space, we can replicate these acceleration techniques over spaces of probability measures and establish the equivalence of convergence rates in these settings. Formal statements of our main results are given below.

\paragraph{\bf Heavy-ball flow.}
The heavy-ball method was introduced by Polyak in 1960~\citep{Po:1964some}. 
In the continuous-time setting, the particle moves along the trajectory defined by
\begin{equation}\label{eqn:hb_ODE1}
\dot{x} = v \,, \qquad 
\dot{v} = -a v - \nabla f(x) \,,
\end{equation}
where $a>0$ is a user-defined parameter, independent of $t$.
The method has $o(1/t)$ convergence for convex objective functions, and has faster convergence than gradient flow for strongly convex objectives, changing the rate from modulus of convexity $m$ for gradient flow to $\sqrt{m}$ for heavy ball, whenever $m\in(0,1)$ is small. 

As we discuss below in Section~\ref{sec:heavy-ball}, analogous dynamics in the space of probability measures are captured by the following {\em heavy-ball flow} equation 
\begin{equation}\label{eqn:intro_heavy-ball-pde}
\partial_t \mu_t + \nabla_x\cdot\left( v \mu_t \right) - \nabla_v\cdot\left( \left(a v + \nabla_x\frac{\delta E}{\delta \rho}[\mu_t^X]\right)\mu_t\right) = 0 \,,\tag{\textsf{HBF}}
\end{equation}
where $\mu_t^X$ denotes the $x$-marginal of $\mu_t$. 
Note the similarity between the coefficients of \eqref{eqn:hb_ODE1} and \eqref{eqn:intro_heavy-ball-pde}. 
The theoretical guarantee is also the same.
\begin{mainthm}
Let $E:\Pc(\Rb^d)\to\Rb$ be 
convex along $2$-Wasserstein geodesics (see Definition~\ref{def:convexity}) and let $\mu_t$ be a solution 
to the heavy-ball flow equation \eqref{eqn:intro_heavy-ball-pde}. Then its $x$-marginal $\mu_t^X$ satisfies
\begin{equation}\label{eqn:heavy-ball-convex_rate1}
    E[\mu^X_t] - E_\ast \leq o\biggl(\frac{1}{t}\biggr) \,.
\end{equation}
Furthermore, if $E:\Pc(\Rb^d)\to\Rb$ is $m$-strongly convex along $2$-Wasserstein geodesics (see Definition~\ref{def:convexity}) and we set $a=2\sqrt{m}$, we have
\begin{equation}\label{eqn:heavy-ball-sconvex_rate1}
    E[\mu^X_t] - E_\ast \leq O(e^{-\sqrt{m}t}) \,.
\end{equation}
\end{mainthm}
Note that the convergence rates in \eqref{eqn:heavy-ball-convex_rate1}-\eqref{eqn:heavy-ball-sconvex_rate1} exactly match those of the heavy-ball method~\citep{AtCa:2017asymptotic,WiReJo:2016lyapunov}. 
The rigorous statement of this result is Theorem~\ref{thm:conv_rate_hb_convex}.

\medskip
\paragraph{\bf Variational acceleration flow.}
Variational acceleration methods \citep{WiWiJo:2016variational} include Nesterov acceleration~\citep{Ne:1983method} as a special case. 
Each member of the class is defined by a triplet $(\alpha_t\,,\beta_t\,,\gamma_t)$ that satisfies certain requirements and follows the trajectory defined by the associated Hamiltonian:
\begin{equation}\label{eqn:var_acc_ODE1}
\dot{x} = v \,, \qquad \dot{v} = 
-(\dot{\gamma}_t - \dot{\alpha}_t) v - e^{2\alpha_t+\beta_t}\nabla f(x) \,.
\end{equation}
The method is known to converge with the rate $e^{-\beta_t}$ for convex objective functions $f$. Essentially, this means the method can converge at an arbitrary rate, given that $\beta_t$ can be chosen to be any rapidly increasing function of $t$.

In Section~\ref{sec:var-acc}, we analyze the counterpart of this approach in the space of probability measures, which we term as the {\em variational acceleration flow} equation:
\begin{equation}\label{eqn:intro_hamiltonian_pde_var}
\partial_t \mu_t + \nabla_x\cdot\left( v \mu_t \right) - \nabla_v\cdot\left(\left((\dot{\gamma}_t-\dot{\alpha}_t)v + e^{2\alpha_t +\beta_t} \nabla_x \frac{\delta E}{\delta \rho}[\mu^X_t] \right)\mu_t\right) = 0
\,.\tag{\textsf{VAF}}
\end{equation}
Note once again the similarity between the coefficients of \eqref{eqn:intro_hamiltonian_pde_var} and \eqref{eqn:var_acc_ODE1}. We establish the following convergence result.

\begin{mainthm}\label{thm:intro_var_convergence}
Let $E:\Pc(\Rb^d)\to\Rb$ be convex along $2$-Wasserstein geodesics (see Definition~\ref{def:convexity}) and let $\mu_t$ be a solution
to the Hamiltonian flow~\eqref{eqn:intro_hamiltonian_pde_var}. 
If the optimal scaling conditions \eqref{eqn:optimal_scale} hold, then the $x$-marginal $\mu_t^X$ satisfies
\begin{equation}\label{eqn:var_acc_rate1}
    E[\mu^X_t] - E_\ast \leq O(e^{-\beta_t}) \,.
\end{equation}
\end{mainthm}
As before, the rate of convergence in~\eqref{eqn:var_acc_rate1} exactly matches that of the corresponding class of methods in Euclidean space~\citep{WiWiJo:2016variational}. The rigorous statement can be found in Theorem~\ref{thm:conv_rate_var}.

\subsection{Summary of Related Work}

We identify two types of research results most relevant to the current paper: (1) Acceleration optimization methods (first-order momentum-based methods) on Euclidean space, and (2) accelerated methods on manifolds and for Bayesian sampling, a problem that shares many characteristics with ours.

Momentum-based type methods achieve acceleration by including an artificial \emph{momentum} or \emph{velocity} variable. 
Notable examples include the heavy-ball method~\citep{Po:1964some} and the Nesterov's accelerated method~\citep{Ne:1983method}, which has the optimal convergence rate for convex functions \citep{Ne:2003introductory} and strongly convex functions \citep{NeYu:1983problem}. 
Traditionally studied in the discrete-in-time setting, recent years have seen investigations of their continuous-time counterparts \citep{AtAl:2000heavy,CaEnGa:2009long,AtCa:2017asymptotic,AtChPeRe:2018fast,SuBoCa:2014differential,ShDuJoSu:2021understanding,KrBaBa:2015accelerated,WiWi:2015accelerated,WiReJo:2016lyapunov,BeJoWi:2018symplectic,MuJo:2019dynamical,DiJo:2021generalized,ScRoBadA:2017integration,MoTaBa:2022systematic,PoSh:2017lyapunov,AlOr:2014linear,ZhMoSr:2018direct,dAScTa:2021acceleration,MaPaTeODDo:2018hamiltonian,FrSuRoVi:2020conformal}. 
Continuous-time analyses typically employ a Lyapunov function~\citep{PoSh:2017lyapunov}. 
In~\citet{WiWiJo:2016variational}, the authors found that the introduction of the momentum variable allows one to achieve an arbitrarily high order of convergence, either through a special design of the Hamiltonian or through the time-dilation technique.

Accelerating convergence of first-order methods over the space of probability measures has yet to attract considerable interest, despite the evident importance of this optimization problem in machine learning applications. 
Topics related to this issue are discussed in \citet{DwChWaYu:2018log,ChChBaJo:2018underdamped,ShLe:2019randomized,LuLuNo:2019accelerating,GaMo:2022semi,ChLiZh:2020wasserstein,LiZhChZhZh:2019understanding,MaChChFlBaJo:2021there,TaMe:2019accelerated,WaLi:2022accelerated,ZhChLiBaEr:2023improved}. 
\citet{LiZhChZhZh:2019understanding} proposes a framework for a class of accelerated Riemannian optimization algorithms over the probability manifold $\Pc(\Rb^d)$.
Momentum-based acceleration methods are formulated as optimal control problems in \citet{TaMe:2019accelerated} and a Lyapunov function is derived by drawing upon the analogy to classical methods. 
In the context of Bayesian sampling, \citet{MaChChFlBaJo:2021there} adopts the perspective of extending probability measures to having support on the phase space and formulates the underdamped Langevin dynamics as a flow over the extended space. 
Convergence of the flows is proved under an assumption that the log-Sobolev inequality holds.

Among the papers referenced, we identify \citet{ChLiZh:2020wasserstein} and \citet{WaLi:2022accelerated} as the ones related most closely to our work. 
Both papers leverage the second-order differential structure over the manifold of probability measures. 
\citet{ChLiZh:2020wasserstein} directly formulate the second-order differential equation, and \citet{WaLi:2022accelerated} adopt a strategy involving the introduction of a Hamiltonian flow across the tangent bundle of the probability measure space.
  
Both these works build on the definition of ``Hamiltonian flow," so on the surface, they are quite similar to ours. 
However, there are important differences with our work. 
Specifically, both studies develop their flow on the \emph{physical space}, focusing on the quantity $\rho(t,x)$. 
In contrast, our approach introduces a distribution over the \emph{phase space}, with the PDE spanning the entirety of $\mu(t,x,v)$. This shift in perspective results in two major consequences:
\begin{itemize}
\item There is no immediate well-posedness theory for the PDEs developed in \citet{ChLiZh:2020wasserstein} and \citet{WaLi:2022accelerated}: their PDEs may not have unique solutions. Stringent regularity assumptions were imposed in a different work by one of the authors of our article~\citet{CaChTs:2019convergence} to ensure existence of a unique solution. 
On the contrary, the Hamiltonian flow PDE of our present paper is \emph{guaranteed} to have a unique solution~\citet{AmGa:2008hamiltonian}. 
Similarly, while~\citet{WaLi:2022accelerated} do provide a Lyapunov analysis of their flow, their proof relies on a smooth optimal transport map: It assumes that both the target distribution and the flow solution have Lebesgue densities. 
In comparison, we work directly on the transport plan and thus can circumvent the regularity assumption.

We believe these improvements to have mathematical depth and to resonate with the comparison between the compressible Euler equation and the Boltzmann equation. 
      The Euler equation serves as the counterpart of the Boltzmann equation on the fluid dynamics side. 
      While Euler develops blow-up singularities, the Boltzmann equation is well-posed~\cite{lions_Boltzmann}. 
      By adding velocity to the unknowns, PDE solutions can span out the singularities to form a regular solution.

\item Another important consequence of this difference in formulations lies in the particle representation. 
      In deriving the equations for $\rho(t,x)$, \cite{ChLiZh:2020wasserstein} and \cite{WaLi:2022accelerated} relied on the so-called \emph{mono-kinetic ansatz}, implying that the velocity $v(t,x) = \nabla\phi(t,x)$ lives on the tangent bundle, and is a function of the space variable $x$. 
      Our approach does not make this assumption, since $v$ is an independent variable, allowing different particles at the same location $x$ to have different velocities. Essentially, while \cite{ChLiZh:2020wasserstein} and \cite{WaLi:2022accelerated} track only the bulk velocity, we allow particles the freedom to roam with individual velocities. 
      Such a conceptual difference resonates with the improvement by the Underdamped Langevin Monte Carlo (ULMC)~\citep{CaLuWa:2023explicit,MaChChFlBaJo:2021there} over the overdamped Langevin Monte Carlo (LMC), where ULMC allows particles to adjust the velocity according to the Hamiltonian, and the PDE is formulated on phase space.
 
\end{itemize}

\cite{Ta:2023accelerated} explores an extension of Nesterov's accelerated method over the space of probability measures. 
The results in that paper are based on a notion of convexity called \emph{transport convexity}, which differs from \emph{geodesic convexity} considered in our work.
This notion can be difficult to verify for several commonly used geodesically convex functionals, including the KL divergence. 
Additionally, \cite{Ta:2023accelerated} does not provide a convergence rate for the heavy-ball method.

\subsection{Organization of the paper}
There are two main technical components of the paper. 
The first is the Wasserstein metric and its induced flow and convexity, while the second concerns Hamiltonian flow methods developed for accelerating optimization in Euclidean space. 
We review these techniques in Section~\ref{sec:prelim}. 
Section~\ref{sec:HF_general} presents our major contributions; We present the Hamiltonian flow PDE in its most general form, and describe the two examples: the heavy-ball method and the variational acceleration flow.
The convergence rates of these methods are discussed in Theorem~\ref{thm:conv_rate_hb_convex} and Theorem~\ref{thm:conv_rate_var}, respectively. Section~\ref{sec:heavy-ball} and Section~\ref{sec:var-acc} are dedicated to the proof of the two theorems.

\section{Background knowledge}\label{sec:prelim}

This section outlines notions relevant to this paper from our two fundamental building blocks: the Wasserstein metric for quantifying distances between probability measures and its induced convexity, and the Hamiltonian flow that guides the dynamics of particles to achieve acceleration. 
(Readers familiar with these topics can skip this section.)

\subsection{Hamiltonian flows}

The idea of accelerating convergence in the space of probability measures arises from the fact that Hamiltonian flows accelerate classical optimization methods in the {\em Euclidean} space. 
For the latter, we consider the minimization problem
\begin{equation}\label{eqn:classical_min}
    x_\ast\in\argmin_{x\in\Rb^d} f(x),
\end{equation}
where $f:\Rb^d\to\Rb$ is a sufficiently smooth convex objective function. We denote the optimal value by $f_\ast=f(x_\ast)$.

A function $f:\Rb^d\to\Rb$ is {\em  $m$-strongly convex} if
\begin{equation}\label{eqn:euc_convex}
f(y)\geq f(x) +\langle\nabla f(x),y-x\rangle +\frac{m}{2}\|y-x\|^2\,,\quad \mbox{for all $x,y\in\Rb^d\,.$}
\end{equation}
The parameter $m \ge 0$ is called the {\em modulus of convexity}.
When $m=0$, we recover the standard convexity condition.
An equivalent definition is that for all $x, y \in \Rb^d$, we have
\begin{equation}\label{eqn:euc_convex_intp}
f(tx+(1-t)y)\leq tf(x)+(1-t)f(y) -\frac{m}{2}t(1-t)\|x-y\|^2, \quad\forall t\in[0,1]\,.
\end{equation}

The most basic first-order strategy for finding the optimal point is the gradient descent method, from which Gradient Flow equation is derived:
\begin{equation}\label{eqn:GD}
    \dot{x}
    =-\nabla f(x)\,.\tag{\textsf{GF}}
\end{equation}
It is well known \citep[see for example][]{PoSh:2017lyapunov}
that \eqref{eqn:GD} converges with the rate
\begin{equation}\label{eqn:GD_convergence}
\begin{cases}
\quad f(x(t)) - f_\ast \leq O\bigl(t^{-1}\bigr) \quad &\text{for convex}\, f\,, \\
\quad f(x(t)) - f_\ast \leq O\bigl(e^{-2m t}\bigr) \quad &\text{for $m$-strongly convex $f$}\,.
\end{cases}
\end{equation}
There are many ways to speed up these convergence rates, and Hamiltonian flows provide a path to do so. 
This approach adds to the position $x$ a velocity $v$ and evolves $(x,v)$ according to a Hamiltonian trajectory.
Defining the Hamiltonian $h_t:\Rb^d\times\Rb^d\to\Rb$ so that
\begin{equation}\label{eqn:ode_flow_hamiltonian}
h_t(x,v) = k_t(v) + p_t(x)\,,
\end{equation}
with $k_t$ and $p_t$ termed the kinetic and potential energy, respectively, the Hamiltonian trajectory is defined by \eqref{eqn:Hamiltonian_ODE1}, restated here:
\begin{equation}\label{eqn:hamiltonian_ODE}
\dot{x} = \nabla_v h_t(x,v)\,, \qquad \dot{v} = -\nabla_x h_t(x,v)\,.
\end{equation}
By selecting  carefully $h_t$---specifically $k_t$ and $p_t$---one can show that the sample following \eqref{eqn:hamiltonian_ODE} converges to $x_\ast$ with accelerated speed. 
We define the two most famous examples of methods in this class.

\begin{example}[Heavy-ball ODE~\citep{Po:1964some}]\label{ex:heavy-ball}
When we set
\begin{equation}
    k_t(v) = \frac{e^{-a t}}{2}\|v\|^2\,, \qquad
    p_t(x) = e^{a t} f(x) \,,
\end{equation}
the Hamiltonian flow \eqref{eqn:hamiltonian_ODE} becomes
\begin{equation}\label{eqn:hamiltonianODE_example}
\dot{x} = e^{-a t} v \,,\qquad \dot{v} = - e^{a t} \nabla f(x)\,.
\end{equation}
Via a change of variable and definition of the scaled velocity $u=e^{-a t}v$, we obtain
\begin{equation}\label{eqn:hamiltonianODE_example_HB-damping}
    \dot{x} = u\,, \qquad
    \dot{u} = - a u - \nabla f(x).
\end{equation}
Compared to~\eqref{eqn:GD_convergence}, the heavy-ball method speeds up the convergence of the gradient flow \eqref{eqn:GD} in both convex and $m$-strongly convex cases. Specifically, we have \citep[see][]{AtCa:2017asymptotic,WiReJo:2016lyapunov} that
\begin{equation}\label{eqn:HB_convergence}
\begin{cases}
\quad f(x(t)) - f_\ast \leq o\bigl(t^{-1}\bigr) \quad &\text{for convex $f$, when we set $a>0$}\,,\\ 
\quad f(x(t)) - f_\ast \leq O\bigl(e^{-\sqrt{m} t}\bigr) \quad &\text{for $m$-strongly convex $f$, when we set $a = 2\sqrt{m}$}\,.
\end{cases}
\end{equation}
\end{example}

\medskip

\begin{example}[Variational acceleration~\citep{WiWiJo:2016variational}]\label{ex:var_acc}
Variational acceleration methods give rise to a large class of algorithms proposed in~\cite{WiWiJo:2016variational} that deploy the following kinetic and potential energy:
\begin{equation}
    k_t(v) = \frac{e^{\alpha_t-\gamma_t}}{2}\|v\|^2\,, \qquad
    p_t(x) = e^{\alpha_t+\beta_t+\gamma_t} f(x) \,,
\end{equation}
where $\alpha_t, \beta_t, \gamma_t$ are time-dependent user-defined parameters. For this definition of the Hamiltonian, the flow is 
\begin{equation}\label{eqn:hamiltonianODE_example2}
        \dot{x} = e^{\alpha_t-\gamma_t} v \,, \qquad
        \dot{v} = - e^{\alpha_t+\beta_t+\gamma_t} \nabla f(x)\,.
\end{equation}
Defining the scaled velocity $u=e^{\alpha_t-\gamma_t}v$, these equations become
\begin{equation}\label{eqn:hamiltonianODE_example2_damping}
\dot{x} = u \,, \qquad
\dot{u} = (\dot{\alpha}_t - \dot{\gamma}_t) u - e^{2\alpha_t+\beta_t}\nabla f(x)\,.
\end{equation}
Under mild assumptions,
it was proved in~\cite{WiWiJo:2016variational} that the dynamics speed up the convergence of \eqref{eqn:GD} when $f$ is convex, the new rate being
\begin{equation}\label{eqn:VA_convergence}
f(x(t)) - f_\ast \leq O\bigl(e^{-\beta_t}\bigr)\,.
\end{equation}
One special example within this framework is the Nesterov acceleration method, which chooses $\alpha_t=\log(2/t)$, $\beta_t=\log(t^2/4)$, and $\gamma_t=2\log(t)$ and yields
\begin{equation}\label{eqn:hamiltonianODE_example2_nesterov}
\dot{x} = u \,, \qquad
\dot{u} = -\frac{3}{t} u - \nabla f(x)\,.
\end{equation}
For this approach, we obtain
\[
f(x(t)) - f_\ast \leq O\bigl(t^{-2}\bigr)\,.
\]
\end{example}

\subsection{Wasserstein metrics and induced convexity}

The set of probability measures forms a nonlinear manifold. 
To quantify the distance between two distributions,  the standard $L_2$ norm inherited from the Hilbert space is insufficient.
Instead, we use techniques developed for Riemannian metrics~\citep{Ot:2001geometry}. We present the main concepts here, omitting details.

Denoting by $\Pc_2(\Rb^d)$ the collection of all probability measures supported on $\Rb^d$ that have finite second moment, we have the following definition of the 2-Wasserstein distance.
\begin{definition} \label{def:wasser}
Given two probability measures $\rho_1, \rho_2\in\Pc_2(\Rb^d)$, the 2-Wasserstein distance $W_2$ between them is defined by
\begin{equation}\label{eqn:wasserstein}
    W_2^2(\rho_1,\rho_2) = \inf\Biggl\{ \int_{\Rb^d\times\Rb^d} \|x-y\|^2 \, \gamma(\rmd x\rmd y) \;:\; \gamma\in\Gamma(\rho_1,\rho_2)\Biggr\}\,,
\end{equation}
where
\[
	\Gamma(\rho_1,\rho_2) = \Bigl\{ \gamma\in\Pc_2(\Rb^d\times\Rb^d) \;:\; (\pi^1)_\sharp \gamma = \rho_1,\;\; (\pi^2)_\sharp \gamma = \rho_2 \Bigr\}\,
\]
denotes the collection of all couplings between $\rho_1$ and $\rho_2$. Here, $(\pi^i)_\sharp \gamma, i=1,2$ denotes the $i$-th marginal of the coupling measure $\gamma$. We denote by $\Gamma_\rmo(\rho_1,\rho_2)\subset\Gamma(\rho_1,\rho_2)$ the collection of optimal couplings that attain the minimum in~\eqref{eqn:wasserstein}.
\end{definition}
Note that $\Gamma_\rmo$ is always non-empty~\citep{Vi:2009optimal}.
According to Brenier's theorem~\citep{Br:1991polar}, when the marginal measure $\rho_1$ (or $\rho_2$) has Lebesgue density, the optimal coupling $\gamma_\rmo$ is unique and is induced by a unique transport map $T:\Rb^d\to \Rb^d$, that is, $\gamma_\rmo = (id\times T)_\# \rho_1$.

The 2-Wasserstein distance induces a (formal) Riemannian structure~\citep{Ot:2001geometry} onto $\Pc_2(\Rb^d)$. 
On a Riemannian manifold, the notion of a gradient can be defined through the underlying metric, giving rise in our case to the 2-Wasserstein gradient: For any functional $E:\Pc_2(\Rb^d)\to\Rb$,
\[
\nabla_{W_2}E[\rho] = -\nabla_x\cdot\left(\rho\nabla_x\frac{\delta E}{\delta\rho}\right)\,.
\]
This notion of gradient allows us to define Wasserstein gradient flows in $\Pc_2(\Rb^d)$, resembling gradient flows in the Euclidean space. 
By guiding the evolution of a probability measure along the steepest descent direction, we define the Wasserstein gradient flow by
\begin{equation}\label{eqn:GF_wasserstein}
\partial_t\rho = -\nabla_{W_2}E[\rho] = \nabla_x\cdot\left(\rho\nabla_x\frac{\delta E}{\delta\rho}\right)\,.\tag{\sf{WGF}}
\end{equation}

With Definition~\ref{def:wasser} of the distance between probability measures, the concept of convexity needs to be rephrased accordingly.
\begin{definition}\label{def:convexity}
For $m\geq 0$, a functional $E:\Pc_2(\Rb^d)\to\Rb$ is called $m$-strongly convex if for every $\rho_1,\rho_2\in\Pc_2(\Rb^d)$ and $\gamma_\rmo\in\Gamma_\rmo(\rho_1,\rho_2)$, we have
\begin{equation}\label{eqn:convexity_diff}
    E[\rho_2] \geq E[\rho_1] + \iint_{\Rb^d\times\Rb^d} \left\langle \nabla_x \frac{\delta E}{\delta \rho}[\rho_1](x), y-x \right\rangle \rmd \gamma_\rmo(x,y) + \frac{m}{2} W_2^2(\rho_1,\rho_2) \,,
\end{equation}
When  $E$ satisfies~\eqref{eqn:convexity_diff} with $m=0$, we say that $E$ is (geodesically) convex.
\end{definition}

Note the resemblance between Definition~\ref{def:convexity} and strong convexity in Euclidean space \eqref{eqn:euc_convex}.
Similarly, extending from the equivalent formulation of strong convexity in \eqref{eqn:euc_convex_intp}, we should also expect that $E$ evaluated at a point on an interpolation between $\rho_1$ and $\rho_2$ should satisfy similar conditions. Indeed, considering the geodesic curve
\begin{equation}
    \rho^{1\to2}_t = ((1-t)\pi^1+t\pi^2)_\sharp \gamma_\rmo \,,\qquad\gamma_\rmo\in\Gamma_\rmo(\rho_1,\rho_2),
\end{equation}
that connects $\rho_1$ and $\rho_2$, with $\rho_0^{1\to2}=\rho_1$ and  $\rho_1^{1\to2}=\rho_2$, Definition~\ref{def:convexity} can be equivalently seen as requiring
\begin{equation}\label{eqn:convexity}
    E[\rho^{1\to2}_t] \leq (1-t) E[\rho_1] + t E[\rho_2] - \frac{m}{2} t(1-t) W_2^2(\rho_1,\rho_2),\quad \mbox{for all $t\in [0,1]$.}
\end{equation}

One interesting class of convex functionals is obtained by extending convex potentials. 
Given a potential function $V$ that is ($m$-strongly) convex on Euclidean space, its associated potential energy
$\Vc:\Pc_2(\Rb^d)\to\Rb$ defined by
\begin{equation}\label{eqn:potential_energy}
    \Vc[\rho] = \int_{\Rb^d} V(x)\, \rho(\rmd x) \,,
\end{equation}
is $m$-strongly convex on $\Pc_2(\Rb^d)$. Another class of convex functionals comes from measuring the KL divergence against a log-concave reference probability measure $\rho_\ast$, that is,
\begin{equation}\label{eqn:KL_div}
    E[\rho]=KL(\rho\,||\,\rho_\ast) = \int_{\Rb^d} \rho(x) \log \frac{\rho(x)}{\rho_\ast(x)}\, \rmd x \,,
\end{equation}
If $\rho_\ast$ is ($m$-strongly) log-concave, then $E$ is ($m$-strongly) convex on $\Pc_2(\Rb^d)$.
That is, if the reference measure takes the form $\rho_\ast \propto e^{-g}$ for some function $g:\Rb^d\to\Rb$, then the ($m$-strong) log-concavity of $\rho_\ast$ is equivalent to the ($m$-strong) convexity of $g$ \citep{AmGiSa:2005gradient}. 

In Euclidean space, the gradient flow finds a minimizer of a convex function. 
An analogous property holds for the Wasserstein gradient flow \eqref{eqn:GF_wasserstein} whenever $E$ is a ($m$-strongly) convex functional. The convergence behavior of~\eqref{eqn:GF_wasserstein}, as shown by~\cite{AmGiSa:2005gradient} and \cite{ChBa:2018convergence}, is as follows:
\begin{equation}\label{eqn:GF_convergence}
\begin{cases}
\quad E[\rho_t] - E_\ast \leq O\bigl(t^{-1}\bigr) \quad &\text{for convex}\, E\,, \\
\quad E[\rho_t] - E_\ast \leq O\bigl(e^{-2m t}\bigr) \quad &\text{for $m$-strongly convex $E$}\,.
\end{cases}    
\end{equation}
We note the exact match of the convergence rates in comparison to the gradient flow \eqref{eqn:GD} in Euclidean space~\eqref{eqn:GD_convergence}.

\section{Hamiltonian flows for optimizing in the space of probability measures}\label{sec:HF_general}

Building on the tools of the previous section, we are ready to define the Hamiltonian flow over the space of probability measures. 
We first collect all probability measures over the phase space that have finite second moment:
\[
\Pc_2(\Rb^d\times\Rb^d)=\left\{\mu\,:\int |x|^2+|v|^2 \, \rmd\mu(x,v)<\infty\right\}\,.
\]
For all $\mu\in\Pc_2(\Rb^d\times\Rb^d)$, denote by $\mu^V$ and $\mu^X$ the marginal distributions of $x$ and $v$, respectively:
\[
\mu^V(\cdot) = \int_{\Rb^d} \mu(\rmd x, \cdot)\,,\quad \mu^X(\cdot) = \int_{\Rb^d} \mu(\cdot, \rmd v)\,.
\]
In the proofs, we use  notation $\mu_{t,x}$ for the conditional distribution of $\mu_t \in \Pc_2(\Rb^d\times\Rb^d)$, following \citet[Theorem 5.3.1]{AmGiSa:2005gradient}:
$$
\mu_{t,x}(v) \coloneqq \mu_t(v|x)\,.
$$
Extending the Hamiltonian defined in~\eqref{eqn:ode_flow_hamiltonian}, we define the Hamiltonian in the probability measure space  $H_t:\Pc_2(\Rb^d\times\Rb^d)\to\Rb$, having the form
\begin{equation}
    H_t[\mu] = K_t[\mu^V] + P_t[\mu^X]\,,
\end{equation}
where $K_t[\mu^V]$ and $P_t[\mu^X]$ represent the kinetic and potential energy, respectively. 
\begin{remark}
We note that the definition of the Hamiltonian separates kinetic and potential energy, each of which depends  on just one of the $v$-marginal and the $x$-marginal of the distribution. 
It is also possible to define a Hamiltonian that depends on the joint distribution. 
One such possibility was deployed in the underdamped Langevin dynamics~\citep{MaChChFlBaJo:2021there}, where the Hamiltonian is the KL divergence between $\mu$ and the distribution $\rho_\ast\otimes \nu_\ast$. Here $\rho_\ast$ is the target distribution and $\nu_\ast \propto \exp(-|v|^2/2)$ represents the standard Gaussian distribution over the velocity space. As elaborated in~\cite{MaChChFlBaJo:2021there}, the underdamped Langevin dynamics can be viewed as a damped version of our Hamiltonian flow.
\end{remark}

We define the Hamiltonian flow on the space of probability measure as follows.
\begin{definition}[Hamiltonian flow over probability measures] \label{def:ham}
Let $t\mapsto H_t$ be the time-dependent Hamiltonian over $\Pc_2(\Rb^d\times\Rb^d)$. 
A Hamiltonian flow with respect to $H_t$ is a curve $t\mapsto \mu_t$ that satisfies
\begin{equation}\label{eqn:hamiltonian_pde}
        \partial_t \mu_t  + \nabla_x \cdot\left(\mu_t \nabla_v \frac{\delta H_t}{\delta \mu}[\mu_t] \right)
        -\nabla_v \cdot\left(\mu_t \nabla_x \frac{\delta H_t}{\delta \mu}[\mu_t] \right) = 0 \quad\text{in the distributional sense},
\end{equation}
with initial condition $\mu_{t=0}=\mu_0\in\Pc_2(\Rb^d\times\Rb^d)$.
\end{definition}

This definition provides the evolution of measures $t\mapsto\mu_t$. 
Well-posedness and absolute continuity of~\eqref{eqn:hamiltonian_pde} with geodesically convex Hamiltonian and general initial data for this equation have been studied in~\cite{AmGa:2008hamiltonian}. 
We note that our definition of Hamiltonian flow is different from the conventional one; see~\cite{ChLiZh:2020wasserstein}. 
Specifically, our formulation expands $\mu^X$ to the phase space $\mu$ and allows each sample to take on different velocities.
We argue that this flow is physically meaningful, intuitive, and gives a meaningful reason to deploy the Hamiltonian flow~\eqref{eqn:hamiltonian_pde} to evolve the probability to minimize $E$ as shown in the following result.
\begin{proposition}\label{prop:dirac-hamiltonian}
The motion of $\delta_{(x(t),v(t))}$, viewed as a probability measure to optimize $E$, agrees with that of $(x(t),v(t))$, viewed as a sample to optimize $f$, if $E$ and $f$, $H_t$ and $h_t$ are related as follows:
\begin{equation}\label{eqn:hamiltonian_classical}
E[\rho]=\int_{\Rb^d} f\,\rmd\rho\,, \qquad H_t[\mu] = \int_{\Rb^{2d}} h_t\, \rmd \mu \,.
\end{equation}
More precisely, we have the following.
\begin{enumerate}[label={(\arabic*)}]
    \item If $t\mapsto (x(t),v(t))$ solves the Hamiltonian ODE in~\eqref{eqn:hamiltonian_ODE}, then the curve of Dirac measure $t\mapsto \mu_t \coloneqq \delta_{(x(t),v(t))}$ solves the Hamiltonian PDE~\eqref{eqn:hamiltonian_pde}.
    \item If $t\mapsto x(t)$ in~\eqref{eqn:hamiltonian_ODE} converges to $x_\ast\in\argmin_xf$, then the Hamiltonian PDE~\eqref{eqn:hamiltonian_pde} drives the $x$-marginal of $\mu_t$ towards $\delta_{x_\ast}$, a minimizer of $E$. 
\end{enumerate}
\end{proposition} 
\begin{proof}
To prove $(1)$, 
we take an arbitrary $\phi\in C^\infty_\rmc(\Rb^d\times\Rb^d)$ and test it on~\eqref{eqn:hamiltonian_pde}, showing that  the result is zero.
Test $\phi$ on the $\partial_t\mu_t$ term with $\mu_t=\delta_{(x(t),v(t))}$, we obtain
\begin{equation}\label{eqn:hamiltonian_t-test}
\begin{aligned}
\frac{\rmd}{\rmd t}\int\phi\,\rmd\mu_t 
&=\frac{\rmd}{\rmd t} \phi(x(t),v(t)) = \nabla_x\phi\cdot \dot{x}
    + \nabla_v\phi\cdot \dot{v}\,\\
    &=\nabla_x\phi\cdot \nabla_v h_t(x(t),v(t)) - \nabla_v\phi\cdot\nabla_xh_t (x(t),v(t))\,,
    \end{aligned}
\end{equation}
where we used chain rule and applied~\eqref{eqn:hamiltonian_ODE} in the last equation. 
Testing $\phi$ on the other two terms in~\eqref{eqn:hamiltonian_pde}, we obtain
\begin{equation}\label{eqn:hamiltonian_xv-test}
\begin{aligned}
\int \nabla_x \cdot \left(\mu_t \nabla_v \frac{\delta H_t}{\delta \mu}\right) \phi\,\rmd{x}\rmd{v}
        &- \int \nabla_v \cdot\left(\mu_t \nabla_x \frac{\delta H_t}{\delta \mu} \right) \phi\,\rmd{x}\rmd{v} \\
        &= -\int \mu_t\left(\nabla_v \frac{\delta H_t}{\delta \mu}\cdot\nabla_x\phi - \nabla_x \frac{\delta H_t}{\delta \mu}\cdot\nabla_v\phi\right)\rmd{x}\rmd{v}\\
        &=-\nabla_x\phi\cdot\nabla_v\frac{\delta H_t}{\delta\mu}(x(t),v(t)) + \nabla_v\phi\cdot\nabla_x\frac{\delta H_t}{\delta\mu}(x(t),v(t))\,.
\end{aligned}
\end{equation}
The relation~\eqref{eqn:hamiltonian_classical} implies that
\[
H_t[\delta_{(x(t),v(t))}] = h_t(x(t),v(t))\,,\qquad \frac{\delta H_t}{\delta\mu}=h_t\,.
\]
Substituting  into~\eqref{eqn:hamiltonian_xv-test} and summing \eqref{eqn:hamiltonian_t-test} and~\eqref{eqn:hamiltonian_xv-test}, we verify that the result is zero.
Since $\phi$ is arbitrary, we conclude that $t\mapsto\delta_{(x(t),v(t))}$ solves~\eqref{eqn:hamiltonian_pde} in the distributional sense.

To show item $(2)$, we need only note that the relation~\eqref{eqn:hamiltonian_classical} guarantees
\[
    E[\delta^X_{(x(t),v(t))}] = E[\delta_{x(t)}]=f(x(t))\,, \qquad E_\ast = E[\delta_{x_\ast}] = f(x_\ast)=f_\ast\,,
\]
thereby concluding the proof.
\end{proof}

Building on the Hamiltonian flow of Definition~\ref{def:ham}, we provide two examples in the next two subsections. 
Both show an improvement in the convergence rate for the problem of finding an optimal $\rho$. 
Some other examples are collected in Appendix~\ref{app:examples}.

\subsection{Heavy-Ball Flow}
The heavy-ball method is known to converge as $O(e^{-\sqrt{m}t})$ in Euclidean space, for $m$-strongly convex objectives and $o(1/t)$ for convex objectives.
We find the corresponding rates for this algorithm in the probability measure space here.

By analogy to Example~\ref{ex:heavy-ball}, we define the following Hamiltonian for any $\mu\in\Pc_2(\Rb^d\times\Rb^d)$:
\begin{equation}
    H_t[\mu] = K_t[\mu^V] + P_t[\mu^X] = \frac{e^{-a t}}{2} \int_{\Rb^d} \|v\|^2 \rmd \mu^V + e^{a t}E[\mu^X],
\end{equation}
where $a>0$ is a user-defined parameter. 
The Fr\'echet derivative is
\begin{equation}
    \frac{\delta H_t}{\delta \mu}[\mu]
    = \frac{e^{-a t}}{2} \|v\|^2 + e^{a t} \frac{\delta E}{\delta \rho}[\mu^X] \,,
\end{equation}
so the Hamiltonian PDE~\eqref{eqn:hamiltonian_pde}  becomes
\begin{equation}\label{eqn:heavy-ball-pde}
    \partial_t \mu_t  + \nabla_x \cdot\left(\mu_t e^{-a t} v \right)
        -\nabla_v \cdot\left(\mu_t e^{a t} \nabla_x \frac{\delta E}{\delta \rho}[\mu^X_t] \right) = 0\,.
\end{equation}
In~\eqref{eqn:hamiltonianODE_example_HB-damping} we introduced a change of variables. Correspondingly, we define $u=e^{-a t }v$ and denote by $\widetilde{\mu}$ the probability measure over this new variable. 
Under this change of variable, \eqref{eqn:heavy-ball-pde} becomes~\eqref{eqn:intro_heavy-ball-pde}.
We spell out this claim in the following proposition.
\begin{proposition}\label{prop:heavy-ball-damped-prop}
Let $\mu_t$ solve the heavy-ball equation~\eqref{eqn:heavy-ball-pde}. Define the map $T_t(x,v) = (x,e^{-at}v)$ and set $\widetilde{\mu}_t \coloneqq (T_t)_\sharp \mu_t$ to be the pushforward of $\mu_t$ under $T_t$. 
Then $\widetilde{\mu}_t$ solves the equation
\begin{equation}\label{eqn:heavy-ball-pde-damped-prop}
    \partial_t \widetilde{\mu}_t + \nabla_x\cdot\left(\widetilde{\mu}_t u \right) -\nabla_u\cdot\left(\widetilde{\mu}_t \left(a u + \nabla_x\frac{\delta E}{\delta \rho}[\widetilde{\mu}_t^X] \right) \right) = 0\,,
\end{equation}
\end{proposition}
\begin{proof}
To show that $\widetilde{\mu}_t$ satisfies~\eqref{eqn:heavy-ball-pde-damped-prop}, we take an arbitrary $\phi\in C_\rmc^\infty(\Rb^d\times\Rb^d)$, test it on $\widetilde{\mu}_t$, and compute its time derivative to obtain
\[
\begin{aligned}
\frac{\rmd}{\rmd t} \int \phi(x,u)\, \rmd \widetilde{\mu}_t (x,u)
&= \frac{\rmd}{\rmd t} \int \phi(x,e^{-at}v)\, \rmd \mu_t (x,v) \\
&= \underbrace{\int \frac{\rmd}{\rmd s} \phi(x,e^{-as}v)\bigg|_{s=t} \rmd \mu_t (x,v)}_{\text{Term I}} + \underbrace{ \int \phi(x,e^{-at}v) \,\rmd \partial_t \mu_t(x,v) }_{\text{Term II}} \,,
\end{aligned}
\]
where in the first equality we use the definition of pushforward map. Term I can be written as an integral in $\widetilde{\mu}_t$ by direct computation, that is,
\[
\text{Term I} = -a \int e^{-at} v\cdot\nabla_v\phi(x,e^{-at}v)\, \rmd \mu_t (x,v) = -a \int u \cdot \nabla_u\phi(x,u)\, \rmd \widetilde{\mu}_t (x,u) \,,
\]
where we use the pushforward definition again in the second equality. For Term II, we use the fact that $\mu_t$ solves~\eqref{eqn:heavy-ball-pde} to obtain
\[
\begin{aligned}
\text{Term II} 
&= - \int \phi(x,e^{-at}v)\, \rmd (\nabla_x \cdot (\mu_t e^{-at}v)) + \int \phi(x,e^{-at}v) \, \rmd \left(\nabla_v \cdot \left( \mu_t e^{at}\nabla_x \frac{\delta E}{\delta \rho}[\mu_t^X]\right)\right)\\
&= \int e^{-at} v \cdot  \nabla_x\phi(x,e^{-at}v)\, \rmd \mu_t(x,v)
- \int e^{at} \nabla_x \frac{\delta E}{\delta \rho}[\mu_t^X] \cdot e^{-at} \nabla_v\phi(x,e^{-at}v)\, \rmd \mu_t(x,v)\\
&= \int u \cdot  \nabla_x\phi(x,u) \,\rmd \widetilde{\mu}_t(x,u)
- \int \nabla_x \frac{\delta E}{\delta \rho}[\widetilde{\mu}_t^X] \cdot \nabla_u\phi(x,u)\, \rmd \widetilde{\mu}_t(x,u) \,.
\end{aligned}
\]
In the last equality, we use $\widetilde{\mu}_t^X=\mu_t^X$. 
By combining both terms, we arrive at~\eqref{eqn:heavy-ball-pde-damped-prop}.
\end{proof}

The heavy-ball gradient-flow PDE~\eqref{eqn:heavy-ball-pde} speeds up the convergence of~\eqref{eqn:GF_wasserstein} in the same way as the heavy-ball ODE~\eqref{eqn:hamiltonianODE_example} speeds up \eqref{eqn:GD}.
We state the result here and leave the proof to Section~\ref{sec:heavy-ball}.
In the following, we denote by $AC([0,\infty),\Pc_2(\Rb^d\times\Rb^d))$ the space of absolutely continuous curves over $\Pc_2(\Rb^d\times\Rb^d)$~\citep{AmGiSa:2005gradient}.
\begin{theorem}\label{thm:conv_rate_hb_convex}
Let $\mu\in AC([0,\infty),\Pc_2(\Rb^d\times\Rb^d))$ solve the heavy-ball flow~\eqref{eqn:heavy-ball-pde} for $E:\Pc_2(\Rb^d)\to\Rb$.
(Equivalently, let $\tilde{\mu}$ solve~\eqref{eqn:intro_heavy-ball-pde}.)
If $E$ and $\mu_t$ are sufficiently smooth, then the marginal distribution $\mu_t^X$ converges as follows:
\begin{subequations}\label{eqn:HB_flow_convergence}
\begin{align}
&E[\mu_t^X]- E_\ast \leq o\left(\frac{1}{t}\right) && \text{for convex $E$, when we set $a>0$, }\label{eqn:HB_flow_convergence_convex}\\
&E[\mu_t^X]- E_\ast \leq O\left(e^{-\sqrt{m}t}\right) && \text{for $m$-strongly convex $E$, when we set $a=2\sqrt{m}$}\label{eqn:HB_flow_convergence_strong}\,.
\end{align}
\end{subequations}
\end{theorem}
Note that convexity and $m$-strong convexity adopt the new geodesic convexity concept. 

\subsection{Variational Acceleration Flow}
The second example generalizes the variational formulation method in Example~\ref{ex:var_acc}, where the Hamiltonian is chosen to be
\begin{equation}\label{eqn:hamiltonian_var}
    H_t[\mu] = K_t[\mu^V] + P_t[\mu^X] = \frac{e^{\alpha_t -\gamma_t }}{2} \int_{\Rb^d}  \|v\|^2 \rmd\mu^V + e^{\alpha_t +\beta_t +\gamma_t } E[\mu^X].
\end{equation}
Here, the functions $\alpha_t,\beta_t,\gamma_t$ are user-defined parameters that satisfy the optimal scaling conditions
\begin{align}\label{eqn:optimal_scale}
    \dot{\beta}_t \leq e^{\alpha_t},\qquad
    \dot{\gamma}_t = e^{\alpha_t}
\end{align}
By differentiating~\eqref{eqn:hamiltonian_var}, we obtain
\begin{equation}
    \frac{\delta H_t}{\delta \mu}[\mu]
    = \frac{e^{\alpha_t-\gamma_t}}{2} \|v\|^2 
    + e^{\alpha_t +\beta_t +\gamma_t} \frac{\delta E}{\delta \rho}[\mu^X] \,,
\end{equation}
so the associated Hamiltonian PDE is
\begin{equation}\label{eqn:hamiltonian_pde_var}
    \partial_t \mu_t  + \nabla_x \cdot\left(\mu_t e^{\alpha_t-\gamma_t} v \right)
        -\nabla_v \cdot\left(\mu_t e^{\alpha_t +\beta_t +\gamma_t} \nabla_x \frac{\delta E}{\delta \rho}[\mu^X_t] \right) = 0\,.
\end{equation}
Similar to the change of variables performed in~\eqref{eqn:hamiltonianODE_example2_damping}, we set $u=e^{\alpha_t-\gamma_t}v$ and denote $\widetilde{\mu}$ as the measure defined over the new variables $(x,u)$. 
Then, \eqref{eqn:hamiltonian_pde_var} becomes \eqref{eqn:intro_hamiltonian_pde_var}. We formalize this claim in the following proposition.
\begin{proposition}\label{prop:var-acc-damped-prop}
Let $\mu_t$ solve the variational acceleration flow equation~\eqref{eqn:hamiltonian_pde_var}. Define the map $T_t(x,v) = (x,e^{\alpha_t-\gamma_t}v)$ and set $\widetilde{\mu}_t \coloneqq (T_t)_\sharp \mu_t$ to be the pushforward of $\mu_t$ under  $T_t$. Then $\widetilde{\mu}_t$ solves the equation
\begin{equation}\label{eqn:var_acc_pde_damped_prop}
    \partial_t \widetilde{\mu}_t  + \nabla_x \cdot\left(\widetilde{\mu}_t u \right)
        -\nabla_v \cdot\left(\widetilde{\mu}_t \left((\dot{\gamma}_t-\dot{\alpha}_t)u + e^{2\alpha_t +\beta_t} \nabla_x \frac{\delta E}{\delta \rho}[\widetilde{\mu}^X_t] \right)\right) = 0.
\end{equation}
\end{proposition}
\begin{proof}
The derivation of~\eqref{eqn:var_acc_pde_damped_prop} from~\eqref{eqn:hamiltonian_pde_var} involves a computation similar to that of Proposition~\ref{prop:heavy-ball-damped-prop}. We omit the details.
\end{proof}

The following set of parameter choice satisfy~\eqref{eqn:optimal_scale}:
\begin{equation}\label{eqn:nesterov_parameters}
\alpha_t = \log(2/t),\qquad \beta_t = \log(t^2/4),\qquad \gamma_t = 2\log(t)\,,
\end{equation}
leading to the Nesterov flow:
\begin{equation}\label{eqn:hamiltonian_pde_nesterov}
    \partial_t \widetilde{\mu}_t  + \nabla_x \cdot\left(\widetilde{\mu}_t u \right)
        -\nabla_v \cdot\left(\widetilde{\mu}_t \left(\frac{3}{t}u + \nabla_x \frac{\delta E}{\delta \rho}[\widetilde{\mu}^X_t] \right)\right) = 0\,.
\end{equation}
As for heavy-ball, we observe the speedup of this variational Hamiltonian flow~\eqref{eqn:hamiltonian_pde_var} compared to the Wasserstein gradient-flow~\eqref{eqn:GF_wasserstein}, with the improvement exactly matching that of variational acceleration method~\eqref{eqn:hamiltonianODE_example2} over \eqref{eqn:GD}.
(The proof of this result is the subject of Section~\ref{sec:var-acc}.)

\begin{theorem}\label{thm:conv_rate_var}
Let the objective functional $E:\Pc_2(\Rb^d)\to\Rb$ be convex along $2$-Wasserstein geodesics. Let $\mu\in AC([0,\infty),\Pc_2(\Rb^d\times\Rb^d))$ solve the Hamiltonian flow~\eqref{eqn:hamiltonian_pde_var} for $E$.
(Equivalently, let $\tilde{\mu}$ solve~\eqref{eqn:intro_hamiltonian_pde_var}.) 
If $E$ and $\mu_t$ are sufficiently smooth and the optimal scaling conditions~\eqref{eqn:optimal_scale} holds,
then the marginal $\mu_t^X$ satisfies
\begin{equation}\label{eqn:convergence_rate_var}
    E[\mu^X_t] - E_\ast \leq O(e^{-\beta_t}) \,.
\end{equation}
Moreover, with coefficients configured as in~\eqref{eqn:nesterov_parameters}, we have
\begin{equation}\label{eqn:convergence_rate_nesterov}
    E[\mu^X_t] - E_\ast \leq O\left(\frac{1}{t^2}\right) \,.
\end{equation}
\end{theorem}

\section{The Heavy-Ball Flow}\label{sec:heavy-ball}

This section describes the convergence of the heavy-ball flow. 
Following some preliminary results, we treat the general convex case followed by the strongly convex case. 

\subsection{Preliminary Results}

First, we consider $\nu \in AC([0,\infty), \Pc_2(\Rb^d))$ that solves
\begin{equation}\label{eqn:preliminary_continuity_eqn}
    \partial_t\nu + \nabla \cdot (\nu \xi) = 0 \,,
\end{equation}
for a given bounded and sufficiently smooth vector field $\xi(t,x)$. Then for any $\sigma\in \Pc_2(\Rb^d)$, we have
\begin{equation}\label{eqn:wass_first_derivative}
\frac{1}{2} \frac{\rmd}{\rmd t} W_2^2(\nu_t,\sigma)
= \int_{\Rb^{d}\times\Rb^{d}} \left\langle x-y, \xi_t(x) \right\rangle \,\rmd \gamma_t(x,y) \,,
\end{equation}
where $\gamma_t\in\Gamma_\rmo(\nu_t,\sigma)$ is an optimal coupling between $\nu_t$ and $\sigma$.

Let $\mu \in AC([0,\infty), \Pc_2(\Rb^d\times\Rb^d))$ solve
\begin{equation}\label{eqn:was_2nd_der_flow_eqn}
\partial_t \mu_t  + \nabla_x \cdot\left(\mu_t F(t,v) \right)
        -\nabla_v \cdot\left(\mu_t G(t,\mu_t,x) \right) = 0 \,.
\end{equation}
Then under some smoothness requirement on $F$ and $G$, we have
\begin{equation}\label{eqn:wass_second_derivative}
\begin{aligned}
\frac{1}{2} \frac{\rmd^+}{\rmd t} \frac{\rmd}{\rmd t} W_2^2(\mu_t^X,\sigma)
\leq &\int_{\Rb^{2d}} \|F(t,v)\|^2\, \rmd \mu_t \\
&+ \int_{\Rb^{3d}} \left\langle x-y, \partial_t F(t,v)  - \nabla_v F(t,v) G(t,\mu_t,x) \right\rangle \, \rmd \mu_{t,x}(v)\, \rmd \gamma_t(x,y) \,,
\end{aligned}
\end{equation}
where $\rmd^+/\rmd t$ denotes the upper derivative at almost every $t\geq0$, $\nabla_v F = (\partial_{v^j} F^i)_{ij}$ denotes the Jacobian, and $\mu_{t,x}$ is the conditional distribution of $\mu_t$. 
The rigorous version of these statements and the smoothness requirements are presented in Appendix~\ref{app:derivatives_proof}.

For the heavy-ball flow~\eqref{eqn:heavy-ball-pde}, we obtain the following result.
\begin{lemma}
Let $\mu\in AC([0,\infty), \Pc_2(\Rb^d\times\Rb^d))$ be a solution of the heavy-ball flow~\eqref{eqn:heavy-ball-pde} and $\sigma\in\Pc_2(\Rb^d)$ be an arbitrary measure. If $E$ and $\mu_t^X$ are sufficiently smooth, we obtain
\begin{equation}\label{eqn:HB-wasserstein-derivative}
\frac{1}{2}\frac{\rmd}{\rmd t} W_2^2(\mu_t^X,\sigma)
= \int_{\Rb^{3d}} \left\langle e^{-at} v, x-y \right\rangle \, \rmd \mu_{t,x}(v)\, \rmd \gamma_t(x,y)\,,
\end{equation}
and
\begin{equation}\label{eqn:HB_wasserstein_second_der}
\begin{aligned}
\frac{1}{2}\frac{\rmd^+}{\rmd t} \frac{\rmd}{\rmd t}W_2^2(\mu_t^X,\sigma) 
\leq &\int_{\Rb^{2d}} \|e^{-at}v\|^2 \rmd \mu_t(x,v) \\
&- \int_{\Rb^{3d}} \left\langle x-y, a e^{-at} v + \nabla\frac{\delta E}{\delta \rho}[\mu_t^X](x) \right\rangle \, \rmd \mu_{t,x}(v) \rmd \gamma_t(x,y) \,.
\end{aligned}
\end{equation}
where $\{\mu_{t,x}\}_{x\in\Rb^d}\subset\Pc_2(\Rb^d)$ is the conditional distribution of $\mu_t$ with respect to its $x$-marginal distribution $\mu_t^X$, and $\gamma_t\in\Gamma_\rmo(\mu_t^X,\sigma)$ is an optimal coupling between $\mu_t^X$ and $\sigma$.
\end{lemma}
\begin{proof}
To prove~\eqref{eqn:HB-wasserstein-derivative}, we apply~\eqref{eqn:wass_first_derivative}. We first derive the continuity equation for $\mu_t^X$ analogously to~\eqref{eqn:preliminary_continuity_eqn} by integrating~\eqref{eqn:heavy-ball-pde} over $v\in\Rb^d$, which gives
\begin{equation}\label{eqn:HB-x-marginal_continuity}
    0 = \partial_t \mu_t^X + \nabla_x\cdot \left(\int_{\Rb^d} e^{-a t} v \mu_t(\cdot,\rmd v) \right) 
    = \partial_t \mu_t^X + \nabla_x\cdot \bigl( e^{-a t} \overline{v}_t(x) \mu_t^X\bigr) \,,
\end{equation}
where $\overline{v}_t(x)=\int_{\Rb^d}v\, \rmd\mu_{t,x}(v)$.
Applying~\eqref{eqn:wass_first_derivative} to~\eqref{eqn:HB-x-marginal_continuity} with $\xi(t,x) = e^{-at}\bar{v}_t(x)$ then gives~\eqref{eqn:HB-wasserstein-derivative}.

To derive~\eqref{eqn:HB_wasserstein_second_der},  we apply~\eqref{eqn:wass_second_derivative} to~\eqref{eqn:heavy-ball-pde} with $F(t,v) = e^{-at}v$ and $G(t,\mu_t^X,x) = e^{at} \nabla\frac{\delta E}{\delta \rho}[\mu_t^X](x)$.
\end{proof}

Furthermore, noting that $\frac{\rmd}{\rmd t}E[\mu_t^X]=\int\frac{\delta E}{\delta \mu^X_t}\, \rmd\partial_t\mu^X_t$, we see that the evolution of $E[\mu_t^X]$ depends on the evolution of $\mu^X_t$. Using~\eqref{eqn:HB-x-marginal_continuity}, we then obtain
\begin{equation}\label{eqn:termIII}
\begin{aligned}
\frac{\rmd}{\rmd t}E[\mu_t^X]
&=\int\frac{\delta E}{\delta \mu^X_t} \frac{\rmd\mu^X_t}{\rmd t} =
\int_{\Rb^{d}} \left\langle \nabla\frac{\delta E}{\delta \rho}[\mu_t^X](x), e^{-at} \bar{v}_t(x) \right\rangle \, \rmd \mu^X_t(x) \\
&= \int_{\Rb^{2d}} \left\langle \nabla \frac{\delta E}{\delta \rho}[\mu_t^X](x),e^{-a t}v\right\rangle \, \rmd \mu_t(x,v)\,.
\end{aligned}
\end{equation}

The proofs below call for repeated use of~\eqref{eqn:HB-wasserstein-derivative},~\eqref{eqn:HB_wasserstein_second_der}, and~\eqref{eqn:termIII}.

\subsection{Convex case}\label{sec:HB-convex}

We show here the convergence rate for the general convex case, stated in \eqref{eqn:HB_flow_convergence_convex}.
As for the Euclidean space analysis, we define a Lyapunov function as follows:
\begin{equation}\label{eqn:def_E}
\Ec_t \coloneqq \frac{1}{2}\int_{\Rb^{2d}}\|e^{-a t} v\|^2 \rmd \mu_t(x,v) + E[\mu_t^X] - E_\ast\,.
\end{equation}
It can be shown that $\Ec_t$ decays in time. 
By differentiating~\eqref{eqn:def_E}, we obtain
\begin{equation}\label{eqn:heavy-ball_Et_diff}
\begin{aligned}
\frac{\rmd}{\rmd t} \Ec_t 
&= -a \int_{\Rb^{2d}}\|e^{-at}v\|^2 \rmd \mu_t +\underbrace{\frac{1}{2}\int_{\Rb^{2d}} \|e^{-a t}v\|^2\frac{\rmd\mu_t}{\rmd t}}_{\text{Term II}} + \underbrace{\frac{\rmd}{\rmd t} E[\mu_t^X]}_{\text{Term III}}\,.
\end{aligned}
\end{equation}
Noting that Term III is already expanded~\eqref{eqn:termIII}, we show that Term II cancels it. 
To see this, we recall~\eqref{eqn:heavy-ball-pde} and use integration by parts in $v$ to obtain
\begin{equation}\label{eqn:termII}
\frac{1}{2}\int_{\Rb^{2d}} \|e^{-a t}v\|^2\frac{\rmd\mu_t}{\rmd t} = -\int_{\Rb^{2d}} \left\langle e^{-2a t}v, e^{a t}\nabla \frac{\delta E}{\delta \rho}[\mu_t^X](x) \right\rangle \, \rmd \mu_t\,.
\end{equation}
Thus~\eqref{eqn:heavy-ball_Et_diff} simplifies to
\begin{equation}\label{eqn:heavy-ball_Et_diff_sum}
\begin{aligned}
\frac{\rmd}{\rmd t} \Ec_t =-a \int_{\Rb^{2d}}\|e^{-at}v\|^2 \rmd \mu_t \leq 0 \,,
\end{aligned}
\end{equation}
showing monotonic decrease of  $\Ec_t$. 
This feature implies that 
\begin{equation}
\int_{\frac{t}{2}}^t \Ec_s\, \rmd s \geq \frac{t}{2}\Ec_t \,, \quad \forall t\geq0 \quad \Longrightarrow \quad E[\mu_t^X]-E_\ast \leq \Ec_t \leq \frac{2}{t}\int_{\frac{t}{2}}^t \Ec_s \rmd s\,.
\end{equation}
Therefore, to show that $E[\mu_t^X]-E_\ast\leq o(t^{-1})$ requires showing that $\int_{t/2}^t\mathcal{E}_s\,\rmd s\to0$  as $t \to \infty$.
In the current case, we can show further that $\Ec_t\in L^1([0,\infty))$, that is $\int_{0}^\infty\mathcal{E}_t\,\rmd t<\infty$, implying that
\[
\lim_{t\to\infty}\int_{t/2}^t\mathcal{E}_s\rmd{s} \leq \lim_{t\to\infty}\int_{t/2}^\infty\mathcal{E}_s\rmd{s} =0\,.
\]
To show $\int_{0}^\infty\mathcal{E}_t\,\rmd t
    <\infty$, we note that
\begin{equation}\label{eqn:ec_t_L1}
    \int_{0}^\infty\mathcal{E}_t\rmd t
    = \underbrace{\int_{0}^\infty \frac{1}{2}\int_{\Rb^{2d}}\|e^{-a t} v\|^2 \rmd \mu_t \, \rmd t}_{\text{Term 1}} 
    + \underbrace{\int_{0}^\infty E[\mu_t^X] - E_\ast \, \rmd t}_{\text{Term 2}}\,.
\end{equation}
For Term 1, we integrate~\eqref{eqn:heavy-ball_Et_diff_sum} over time to obtain
\begin{equation}\label{eqn:app_heavy-ball_kinetic_integrable}
2\times \text{Term 1} = \int_0^\infty\int_{\Rb^{2d}}\|e^{-at}v\|^2\,\rmd \mu_t\, \rmd t = \lim_{s\to\infty}\int_0^s\int_{\Rb^{2d}}\|e^{-at}v\|^2 \,\rmd \mu_t\, \rmd t = \lim_{s\to\infty} \frac{\Ec_0-\Ec_s}{a} \leq \frac{\Ec_0}{a} \,,
\end{equation}
proving finiteness of this term.
For Term 2, we show below that 
\begin{equation}\label{eqn:app_heavy-ball_potential_integrable}
\text{Term 2} = \int_0^\infty E[\mu_t^X] - E_\ast \,\rmd t \leq a \Lc_0\,.
\end{equation}
From \eqref{eqn:ec_t_L1}, we thus have boundedness of $\int_0^\infty\mathcal{E}_t\,\rmd t$ and thus the convergence of $E[\mu_t^X]-E_\ast$ with a rate of $o(t^{-1})$.

To show~\eqref{eqn:app_heavy-ball_potential_integrable}, we define the following Lyapunov function inspired by the proof of the regular heavy-ball method~\citep{AtCa:2017asymptotic}:
\begin{equation}\label{eqn:hb_lya_original}
\Lc_t = \int_{\Rb^{3d}} \frac{1}{2} \left\|x + \frac{e^{-at}}{a}v -y\right\|^2 \rmd \mu_{t,x}(v) \rmd\gamma_t(x,y) + \frac{1}{a^2}\left(E[\mu_t^X] - E_\ast\right) \,,
\end{equation}
where, as above, $\{\mu_{t,x}\}_{x\in\Rb^d}\subset\Pc_2(\Rb^d)$ is the conditional distribution of $\mu_t$ at $x$ and $\gamma_t\in\Gamma_\rmo(\mu_t^X,\rho_\ast)$ is an optimal transport plan. 
By expanding the quadratic term, we obtain
\begin{equation} \label{eq:si1}
    \Lc_t = \underbrace{\int_{\Rb^{2d}} \frac{1}{2} \left\|x -y\right\|^2 \rmd\gamma_t}_{\text{term A}} 
    +  \underbrace{\int_{\Rb^{3d}} \left\langle\frac{e^{-at}}{a}v ,x-y\right\rangle \rmd \mu_{t,x} \rmd\gamma_t}_{\text{term B}} 
    + \underbrace{\int_{\Rb^{2d}} \frac{1}{2} \left\|\frac{e^{-at}}{a}v\right\|^2 \rmd \mu_{t}}_{\text{term C}}
    +\underbrace{\frac{1}{a^2} \left(E[\mu_t^X] - E_\ast\right)}_{\text{term D}}\,.
\end{equation}
We now take the time derivative of each term to finally show that
\begin{equation}\label{eqn:hb_lya_time}
    \frac{\rmd^+}{\rmd t}\Lc_t\leq -\frac{1}{a}(E[\mu_t^X]-E_\ast)\,.
\end{equation}
\begin{itemize}
    \item[{\em Term A.}] Noticing that Term A is equivalent to $\frac{1}{2}W_2^2(\mu_t^X,\rho_\ast)$, we utilize~\eqref{eqn:HB-wasserstein-derivative} to obtain
\begin{equation}
\frac{\rmd}{\rmd t}\text{Term A}= \int_{\Rb^{3d}} \left\langle e^{-at} v, x-y \right\rangle \, \rmd \mu_{t,x}(v) \rmd \gamma_t(x,y)\,.
\end{equation}
\item[{\em Term B.}] From~\eqref{eqn:HB-wasserstein-derivative} with $\sigma=\rho_\ast$, we note that Term B can be written as
\begin{equation*}
\text{Term B} = \frac{1}{a}\int_{\Rb^{3d}} \left\langle e^{-at}v ,x-y\right\rangle \rmd \mu_{t,x}(v) \rmd\gamma_t(x,y) = \frac{1}{2a}\frac{\rmd}{\rmd t}W^2_2(\mu^X_t\,,\rho_\ast) \,.
\end{equation*}
Thus, we have
\begin{equation*}
    \frac{\rmd^+}{\rmd t}\text{Term B}\leq \frac{1}{a}\int_{\Rb^{2d}} \|e^{-at}v\|^2 \rmd\mu_t- \int_{\Rb^{3d}} \left\langle  x-y, e^{-at}v+\frac{1}{a}\nabla\frac{\delta E}{\delta \rho}[\mu_t^X](x)\right\rangle \, \rmd\mu_{t,x} \rmd\gamma_t\,.
\end{equation*}
\item[{\em Term C.}] Differentiating this term in time gives
\begin{equation}
\frac{\rmd}{\rmd t}\text{Term C}=-\frac{1}{a}\int_{\Rb^{2d}} \|e^{-at}v\|^2 \rmd\mu_t
- \int_{\Rb^{2d}} \left\langle \frac{e^{-2at}}{a^2}v, e^{at}\nabla\frac{\delta E}{\delta \rho}[\mu_t^X](x) \right\rangle\, \rmd \mu_t\,,
\end{equation}
where the two terms in this expression come from differentiating $\|e^{-at}v\|$ and $\mu_t$.
\item[{\em Term D.}]  By differentiating in time  and following~\eqref{eqn:termIII}, we obtain
\begin{equation*}
    \frac{\rmd}{\rmd t}\text{Term D} =\frac{1}{a^2}\int_{\Rb^{2d}} \left\langle e^{-2a t}v, e^{a t}\nabla \frac{\delta E}{\delta \rho}[\mu_t^X](x) \right\rangle \, \rmd \mu_t \,.
\end{equation*}
\end{itemize}

By summing Terms A, B, C, and D, we obtain
\begin{equation}\label{eqn:app_heavy-ball-convex_Lt}
\begin{aligned}
\frac{\rmd^+}{\rmd t} \Lc_t
\leq &- \frac{1}{a}\int_{\Rb^{3d}} \left\langle  x-y, \nabla\frac{\delta E}{\delta \rho}[\mu_t^X](x)\right\rangle \, \rmd\mu_{t,x} \rmd\gamma_t\leq -\frac{1}{a}(E[\mu_t^X]-E_\ast) \,,
\end{aligned}
\end{equation}
where we used the convexity of $E$ in the last inequality, concluding our proof of the claim \eqref{eqn:hb_lya_time}.

By integrating~\eqref{eqn:app_heavy-ball-convex_Lt} in time, we obtain
\begin{equation}
\int_0^\infty \left( E[\mu_t^X] - E_\ast \right) \, \rmd t = \lim_{s\to\infty}\int_0^s \left( E[\mu_t^X] - E_\ast \right)\, \rmd t \leq a \lim_{s\to\infty}(\Lc_0 - \Lc_s) \leq   a \Lc_0\,. 
\end{equation}
showing~\eqref{eqn:app_heavy-ball_potential_integrable} and concluding the proof.

\subsection{Strongly Convex Case}

This section is dedicated to showing the convergence rate \eqref{eqn:HB_flow_convergence_strong} for a strongly convex function $f$ with modulus of convexity $m>0$.

We define a Lyapunov function inspired by the one in~\cite{WiReJo:2016lyapunov}:
\begin{equation}\label{eqn:hb_def_lya_strong}
\Lc_t = \left(\int_{\Rb^{3d}} \frac{m}{2} \left\|x + \frac{e^{-2\sqrt{m}t}}{\sqrt{m}}v -y\right\|^2 \rmd \mu_{t,x}(v) \rmd\gamma_t(x,y) + E[\mu_t^X] - E_\ast\right) e^{\sqrt{m}t} \,.
\end{equation}
By expanding the quadratic term, using the same expansion as in \eqref{eq:si1}, we can write
\begin{equation} \label{eq:si2}
\Lc_t = e^{\sqrt{m} t}\left( \underbrace{\frac{m}{2} W_2^2(\mu_t^X,\rho_\ast)}_{\text{Term A}} + \underbrace{\frac{\sqrt{m}}{2}\frac{\rmd}{\rmd t} W_2^2(\mu_t^X,\rho_\ast)}_{\text{Term B}} + \underbrace{\frac{1}{2}\int_{\Rb^{2d}} \|e^{-2\sqrt{m}t}v\|^2 \rmd \mu_t}_{\text{Term C}} + \underbrace{E[\mu_t^X] - E_\ast}_{\text{Term D}} \right)  \,.
\end{equation}
We show below that
$\frac{\rmd^+}{\rmd t}\Lc_t\leq 0$
by analyzing the contributions of the four terms in turn.

\medskip

\begin{itemize}
\item[{\em Term A.}] Using the same strategy as for Term A in Section~\ref{sec:HB-convex} and applying~\eqref{eqn:wass_first_derivative}, we have
\begin{equation}
\frac{\rmd}{\rmd t}\text{Term A}=\frac{m}{2}\frac{\rmd}{\rmd t} W_2^2(\mu_t^X,\rho_\ast)
= m\times\eqref{eqn:HB-wasserstein-derivative},
\end{equation}
for $\gamma_t\in\Gamma_\rmo(\mu_t^X,\rho_\ast)$.
\item[{\em Term B.}] By applying~\eqref{eqn:wass_second_derivative} to the heavy-ball flow PDE~\eqref{eqn:heavy-ball-pde} with $a=2\sqrt{m}$, we obtain an upper bound for the second-order derivative, as follows:
\begin{equation*}
\begin{aligned}
\frac{\rmd^+}{\rmd t}\text{Term B}
\leq\frac{\sqrt{m}}{2}\frac{\rmd^+}{\rmd t} \frac{\rmd}{\rmd t}W_2^2(\mu_t^X,\rho_\ast) =\sqrt{m}\times\eqref{eqn:HB_wasserstein_second_der}\,.
\end{aligned}
\end{equation*}
\item[{\em Term C.}] By differentiating Term C with respect to $t$, we obtain
\begin{equation*}
\begin{aligned}
\frac{\rmd}{\rmd t}\text{Term C}=-2\sqrt{m}\int_{\Rb^{2d}} \|e^{-2\sqrt{m}t}v\|^2 \rmd\mu_t
- \int_{\Rb^{2d}} \left\langle e^{-2\sqrt{m}t}v, \nabla\frac{\delta E}{\delta \rho}[\mu_t^X](x) \right\rangle\, \rmd \mu_t\,.
\end{aligned}
\end{equation*}
\item[{\em Term D.}] Here we recall \eqref{eqn:termIII}.
\end{itemize}

\medskip

By assembling Terms A, B, C, and D,  and substituting into \eqref{eq:si2}, we obtain
\begin{equation}
\begin{aligned}
\frac{\rmd^+}{\rmd t} \Lc_t=& \sqrt{m}e^{\sqrt{m}t} \left(\frac{m}{2}W_2^2(\mu_t^X,\rho_\ast) + E[\mu_t^X] - E_\ast -\int_{\Rb^{2d}}\left\langle 
x-y,\nabla\frac{\delta E}{\delta \rho}[\mu_t^X](x) \right\rangle \, \rmd \gamma_t\right) \\
&-\frac{\sqrt{m}e^{\sqrt{m}t}}{2}\int_{\Rb^{2d}} \|e^{-2\sqrt{m}t}v\|^2 \rmd\mu_t \\
\leq& \sqrt{m}e^{\sqrt{m}t} \underbrace{\left(\frac{m}{2}W_2^2(\mu_t^X,\rho_\ast) + E[\mu_t^X] - E_\ast -\int_{\Rb^{3d}}\left\langle 
x-y,\nabla\frac{\delta E}{\delta \rho}[\mu_t^X](x) \right\rangle \, \rmd \gamma_t\right)}_{\leq 0}\leq 0 \,.
\end{aligned}
\end{equation}
The last inequality follows from the strong convexity of $E$ (Definition~\ref{def:convexity}).

Recalling~\eqref{eqn:hb_def_lya_strong}, we obtain that 
\begin{equation}
e^{\sqrt{m}t}(E[\mu_t^X]-E_\ast)\leq\Lc_t\leq\Lc_{t=0} \; \Rightarrow \; E[\mu_t^X]-E_\ast\leq O(e^{-\sqrt{m}t})\,,
\end{equation}
concluding the proof.

\section{Variational Acceleration Flow}\label{sec:var-acc}

Here we study the convergence rate of variational acceleration flow.

To prepare, we integrate~\eqref{eqn:hamiltonian_pde_var} in $v$ to obtain the evolution of the $x$-marginal $\mu_t^X$:
\begin{equation}\label{eqn:x-marginal_continuity}
    0 = \partial_t \mu_t^X + \nabla_x\cdot \left(\int_{\Rb^d} e^{\alpha_t-\gamma_t} v \mu_t(\cdot,\rmd v) \right) 
    = \partial_t \mu_t^X + \nabla_x\cdot \left( e^{\alpha_t-\gamma_t} \overline{v}_t(x) \mu_t^X\right) \,,
\end{equation}
where $\overline{v}_t(x)=\int_{\Rb^d}v \rmd\mu_{t,x}(v)$ and $\{\mu_{t,x}\}_{x\in\Rb^d}\subset\Pc_2(\Rb^d)$ is the conditional distribution, given $\mu_t^X$.

\subsection{Proof of Theorem~\ref{thm:conv_rate_var}}

First, deploying the strategy of the previous section, we define a Lyapunov function
\begin{equation}\label{eqn:def_va_lya}
\Lc_t \coloneqq \frac{1}{2} \int_{\Rb^{3d}} \|x+e^{-\gamma_t} v - y\|^2 \rmd \mu_{t,x}(v)\rmd\bar{\gamma}_t(x,y) + e^{\beta_t} (E[\mu_t^X] - E_\ast)\,,
\end{equation}
where $\bar{\gamma}_t\in\Gamma_\rmo(\mu_t^X,\rho_\ast)$ is the optimal transport plan between $\mu_t^X$ and $\rho_\ast$. We show below that $\frac{\rmd^+}{\rmd t}\Lc_t\leq 0$, which implies that 
\begin{equation}
e^{\beta_t} (E[\mu_t^X]-E_\ast) \leq \Lc_t \leq \Lc_{t=0} \,, \quad \forall t\geq0,
\end{equation}
from which it follows that $E[\mu_t^X]-E_\ast \leq O(e^{-\beta_t})$.

By expanding~\eqref{eqn:def_va_lya}, we obtain 
\begin{equation}\label{eqn:Lyap_func}
    \Lc_t = \frac{1}{2} \underbrace{W_2^2(\mu_t^X,\rho_\ast)}_{\text{Term A}} +  e^{-\alpha_t} \underbrace{\frac{1}{2} \frac{\rmd}{\rmd t} W_2^2 (\mu_t^X, \rho_\ast)}_{\text{Term B}} + \frac{1}{2}e^{-2\gamma_t}\underbrace{\int_{\Rb^{2d}} \|v\|^2 \rmd \mu_t}_{\text{Term C}} + e^{\beta_t} \underbrace{(E[\mu_t^X] - E_\ast)}_{\text{Term D}} \,.
\end{equation}
Following a familiar strategy, we analyze the derivatives of the four terms in turn.

\medskip

\begin{itemize}
    \item[{\em Term A.}] Deploying~\eqref{eqn:wass_first_derivative} in the context of~\eqref{eqn:hamiltonian_pde_var}, we have
\begin{equation}\label{eqn:W2_derivative_var}
 \begin{aligned}
\frac{\rmd}{\rmd t}\text{Term A}&=\frac{1}{2}\frac{\rmd}{\rmd t} W_2^2(\mu_t^X, \rho_\ast) \\
&= \frac{1}{2}\int_{\Rb^{2d}} \left \langle e^{\alpha_t-\gamma_t} \overline{v}_t(x), x-y \right\rangle \, \rmd \bar{\gamma}_t(x,y) \\
&= \frac{1}{2}\int_{\Rb^{3d}} \left \langle e^{\alpha_t-\gamma_t} v, x-y \right\rangle \, \rmd \mu_{t,x}(v) \rmd \bar{\gamma}_t(x,y) \,.
\end{aligned}
\end{equation}
\item[{\em Term B.}] By deploying~\eqref{eqn:wass_second_derivative} in the context of~\eqref{eqn:hamiltonian_pde_var}, we have
\begin{equation}\label{eqn:W2_second_derivative_var}
\begin{aligned}
\frac{\rmd^+}{\rmd t}\text{Term B} &= \frac{1}{2} \frac{\rmd^+}{\rmd t} \frac{\rmd}{\rmd t} W_2^2(\mu_t^X,\sigma)\\
& \leq \int_{\Rb^{2d}} \|e^{\alpha_t-\gamma_t} v\|^2 \rmd \mu_t \\
& \quad - \int_{\Rb^{2d}} \left\langle e^{\alpha_t-\gamma_t} (x-y), (e^{\alpha_t}-\dot{\alpha}_t)\overline{v}_t(x) + e^{\alpha_t+\beta_t+\gamma_t} \nabla \frac{\delta E}{\delta \rho}[\mu_t^X](x) \right\rangle \,\rmd \bar{\gamma}_t(x,y) \,.
\end{aligned}
\end{equation}
\item[{\em Term C.}] By differentiating Term C with respect to time and utilizing~\eqref{eqn:hamiltonian_pde_var}, we have
\begin{equation*}
    \begin{aligned}
        \frac{\rmd}{\rmd t}\text{Term C}=\int \|v\|^2\frac{\rmd \mu}{\rmd t}=2 e^{\alpha_t +\beta_t +\gamma_t} \int \left\langle v, \nabla_x \frac{\delta E}{\delta \rho}[\mu^X_t] \right\rangle \rmd\mu_t\,. 
    \end{aligned}
\end{equation*}
\item[{\em Term D.}] By differentiating Term D with respect to time and utilizing~\eqref{eqn:x-marginal_continuity}, we obtain
\[
\frac{\rmd}{\rmd t}\text{Term D}=\frac{\rmd}{\rmd t}E[\mu_t^X]=\int\frac{\delta E}{\delta \mu^X_t} \frac{\rmd\mu^X_t}{\rmd t} 
 = \int \left\langle e^{\alpha_t-\gamma_t}\overline{v}_t(x), \nabla\frac{\delta E}{\delta\rho}[\mu^X_t] \right\rangle \rmd\mu^X_t\,.
\]
\end{itemize}

By assembling all terms and substituting into the limiting time derivative of \eqref{eqn:Lyap_func}, we obtain
\begin{equation}
\begin{aligned}
\frac{\rmd^+}{\rmd t} \Lc_t
\leq& \dot{\beta}_t e^{\beta_t}(E[\mu_t^X]-E_\ast) - e^{\alpha_t+\beta_t}\int_{\Rb^{2d}} \left\langle  x-y,  \nabla \frac{\delta E}{\delta \rho}[\mu_t^X](x) \right\rangle \,\rmd \bar{\gamma}_t(x,y) \,.
\end{aligned}
\end{equation}
From the optimal scaling condition $\dot{\beta}_t\leq e^{\alpha_t}$, and using convexity of $E$, we obtain that
\begin{equation}
\frac{\rmd^+}{\rmd t} \Lc_t \leq e^{\alpha_t+\beta_t} \left[\underbrace{E[\mu_t^X]-E_\ast - \int_{\Rb^{2d}} \left\langle  x-y,  \nabla \frac{\delta E}{\delta \rho}[\mu_t^X](x) \right\rangle \,\rmd \bar{\gamma}_t(x,y)}_{\leq 0} \right]\leq 0 \,.
\end{equation}
This completes the proof.

\subsection{Time Dilation}\label{sec:hamilton_flow_time_dilation}
We can show, using similar analysis to that of \cite{WiWiJo:2016variational}, that the family of Hamiltonians
defined by \eqref{eqn:hamiltonian_var} is closed under time dilation.
Given a smooth increasing function $\tau:\Rb_+\to\Rb_+$ and a phase probability distribution $\mu_t(x,v)$, we consider the reparameterized curve $\widetilde{\mu}_t = \mu_{\tau(t)}$. 
For clarity in the analysis of this section, we change the notation for the Hamiltonian  of \eqref{eqn:hamiltonian_var}, denoting it by $H_{\alpha,\beta,\gamma}$ to emphasize its dependence on the time-dependent parameters $\alpha$, $\beta$, and  $\gamma$. We have the following result.
\begin{theorem}\label{thm:time_dilation}
    If $\mu_t$ satisfies the Hamiltonian flow equation~\eqref{eqn:hamiltonian_pde_var} for the Hamiltonian $H_{\alpha,\beta,\gamma}$ of \eqref{eqn:hamiltonian_var}, then the reparameterized curve $\widetilde{\mu}_t=\mu_{\tau(t)}$ satisfies the Hamiltonian flow equation for the rescaled Hamiltonian $H_{\widetilde{\alpha},\widetilde{\beta},\widetilde{\gamma}}$, with the modified parameters defined as follows:
    \begin{subequations}\label{eqn:time_dilate_coef}
    \begin{align}
        \widetilde{\alpha}_t = \alpha_{\tau(t)} + \log \dot{\tau}(t),\qquad 
        \widetilde{\beta}_t = \beta_{\tau(t)},\qquad 
        \widetilde{\gamma}_t = \gamma_{\tau(t)} \,.
    \end{align}
    \end{subequations}
    Furthermore, $\alpha$, $\beta$, and $\gamma$ satisfy the optimal scaling conditions~\eqref{eqn:optimal_scale} if and only if $\widetilde{\alpha}$, $\widetilde{\beta}$, and $\widetilde{\gamma}$ satisfy these conditions, respectively.
\end{theorem}
\begin{proof}
    Note that the time derivative of the reparameterized curve is given by $\partial_t \widetilde{\mu}_t = \dot{\tau}(t) \partial_\tau \mu_\tau|_{\tau=\tau(t)}$. 
    Thus, following \eqref{eqn:hamiltonian_pde_var},  the following equation is satisfied by $\widetilde{\mu}_t$:
    \begin{equation}
        \partial_t \widetilde{\mu}_t = 
        - \nabla_x \cdot \left( \mu_{\tau(t)}\dot{\tau}(t) e^{\alpha_{\tau(t)}-\gamma_{\tau(t)}} \right) 
        + \nabla_v\cdot\left( \mu_{\tau(t)} \dot{\tau}(t) e^{\alpha_{\tau(t)}+\beta_{\tau(t)}+\gamma_{\tau(t)}} \nabla \frac{\delta E}{\delta \rho}\left[\mu^X_{\tau(t)}\right] \right) \,.
    \end{equation}
    With the definition \eqref{eqn:time_dilate_coef} of the modified parameters $\widetilde{\alpha}$, $\widetilde{\beta}$, and  $\widetilde{\gamma}$, we can write this equation as
    \begin{equation}
        \partial_t \widetilde{\mu}_t = 
        - \nabla_x \cdot \left( \widetilde{\mu}_t e^{\widetilde{\alpha}_t-\widetilde{\gamma}_t} \right) 
        + \nabla_v\cdot\left( \widetilde{\mu}_t e^{\widetilde{\alpha}_t+\widetilde{\beta}_t+\widetilde{\gamma}_t} \nabla \frac{\delta E}{\delta \rho}\left[\widetilde{\mu}_t^X\right] \right) \,,
    \end{equation}
    which is the Hamiltonian flow equation~\eqref{eqn:hamiltonian_pde_var} for the rescaled Hamiltonian $H_{\widetilde{\alpha},\widetilde{\beta},\widetilde{\gamma}}$.
\end{proof}

The previous theorem is the analog of \cite[Theorem 2.2]{WiWiJo:2016variational}. As observed in~\cite{WiWiJo:2016variational}, a major appeal of this theorem is that it links up a class of methods through time-dilation. 
In particular, set
\[
\alpha_t=\log(2/t), \qquad \beta_t=\log(t^2/4),  \qquad \gamma_t=2\log(t)
\]
as the parameters for Nesterov acceleration, which  achieves convergence of $E[\mu_t]-E_\ast=O(1/t^2)$.
Then for any $p\geq 2$, setting $\tau(t) = t^{p/2}$, we have the accelerated rate  $E[\tilde{\mu}_t]-E_\ast=O(1/t^p)$.

We note that in general, when we reparameterize time by a time-dilation function $\tau(t)$, the Hamiltonian functional transforms to $\widetilde{H}_t[\mu] = \dot{\tau}(t) H_{\tau(t)}[\mu]$. Thus, the result of Theorem~\ref{thm:time_dilation} can be written as 
\begin{equation*}
    H_{\widetilde{\alpha},\widetilde{\beta},\widetilde{\gamma}}(t) = \dot{\tau}(t) H_{\alpha,\beta,\gamma}(\tau(t)) \,,
\end{equation*}
which can be verified by the definition of Hamiltonian~\eqref{eqn:hamiltonian_var} and the modified parameters \eqref{eqn:time_dilate_coef}.

\section{Algorithms and Numerical Experiments}\label{sec:alg_num}

In this section, we report on numerical experiments with the Hamiltonian flows introduced above. 
In Section~\ref{sec:alg} we lay out the algorithms for running~\eqref{eqn:intro_heavy-ball-pde} and~\eqref{eqn:intro_hamiltonian_pde_var} using representative particles, while in Section~\ref{sec:numerical_evidence}, we showcase the application of the algorithms in two specific examples: potential energy and Bayesian sampling. We consider only continuous-time models in this section, deferring the development of discrete-in-time algorithms to future research.

\subsection{Implementation of (\sf{HBF}) and (\sf{VAF})}\label{sec:alg}

To study  Hamiltonian flows, we find numerical solutions $\mu\in AC([0,\infty),\Pc_2(\Rb^d\times\Rb^d))$  of the equation 
\begin{equation}\label{eqn:hamiltonian_pde_num}
\partial_t \mu_t  + \nabla_x \cdot\left(\mu_t \nabla_v \frac{\delta H_t}{\delta \mu}[\mu_t] \right)
        -\nabla_v \cdot\left(\mu_t \nabla_x \frac{\delta H_t}{\delta \mu}[\mu_t] \right) = 0 \,.
\end{equation}
Given that $(x,v)\in\mathbb{R}^d\times\Rb^d$, the classical numerical approach for computing this equation requires discretization over the domain $\Rb^d\times\Rb^d$, which is computationally prohibitive for nontrivial dimensions $d$.  
In this context, a Monte Carlo solver based on particle approximation can be more robust for higher values of $d$.
We define an empirical distribution based on $N$ particles $(X^i,V^i)$, as follows:
\[
\mu_t \approx \overline{\mu}_t = \frac{1}{N} \sum_{i=1}^N \delta_{(X^i(t),V^i(t))}\,.
\]
We denote by $\overline{\mu}^X_t = \frac{1}{N}\sum_{j=1}^N \delta_{X^j(t)}$ the $x$-marginal of the empirical measure. 
By substituting $\overline{\mu}_t$ into~\eqref{eqn:hamiltonian_pde_num} and testing it on $\phi\in C_\rmc^\infty(\Rb^d\times\Rb^d)$, we obtain that
\begin{equation}
\sum_{i=1}^N \nabla_x \phi(X^i,V^i) \left(\dot{X}^i-\nabla_v \frac{\delta H_t}{\delta \mu}[\overline{\mu}_t](X^i,V^i)\right) + \nabla_v \phi(X^i,V^i) \left(\dot{V}^i+\nabla_x \frac{\delta H_t}{\delta \mu}[\overline{\mu}_t](X^i,V^i)\right) = 0 \,,
\end{equation}
which suggests the following equations for evolution of the particles:
\begin{equation}
\dot{X}^i = \nabla_v\frac{\delta H_t}{\delta \mu}[\overline{\mu}_t](X^i,V^i)\,, \quad \dot{V}^i = -\nabla_x \frac{\delta H_t}{\delta \mu}[\overline{\mu}_t](X^i,V^i) \,, \quad i=1,\dots,N \,.
\end{equation}
By substituting the various definitions of $H_t$ considered so far, we arrive at the following flows:
\begin{itemize}
\item Heavy-ball flow~\eqref{eqn:intro_heavy-ball-pde}:
\begin{equation}\label{eqn:particle_hb}
\dot{X}^i = V^i\,, \quad \dot{V}^i = -a V^i - \nabla_x \frac{\delta E}{\delta \rho}[\overline{\mu}^X_t](X^i) \,, \quad i=1,\dots,N \,.
\end{equation}
\item Variational-acceleration-flow~\eqref{eqn:intro_hamiltonian_pde_var} in its general form:
\begin{equation}
\dot{X}^i = V^i\,, \quad \dot{V}^i = -(\dot{\gamma}_t-\dot{\alpha}_t) V^i - e^{2\alpha_t+\beta_t}\nabla_x \frac{\delta E}{\delta \rho}[\overline{\mu}^X_t](X^i) \,, \quad i=1,\dots,N \,.
\end{equation}
\item Nesterov flow as an example of~\eqref{eqn:intro_hamiltonian_pde_var} using the coefficients~\eqref{eqn:nesterov_parameters}:
\begin{equation}\label{eqn:particle_nes}
\dot{X}^i = V^i\,, \quad \dot{V}^i = -\frac{3}{t} V^i - \nabla_x \frac{\delta E}{\delta \rho}[\overline{\mu}^X_t](X^i) \,, \quad i=1,\dots,N \,.
\end{equation}
\item Exponential convergence as an example of~\eqref{eqn:intro_hamiltonian_pde_var} using coefficients $[\alpha_t,\beta_t,\gamma_t]=[0,t,t]$:
\begin{equation}\label{eqn:particle_exp}
\dot{X}^i = V^i\,, \quad \dot{V}^i = - V^i - e^{t} \nabla_x \frac{\delta E}{\delta \rho}[\overline{\mu}^X_t](X^i) \,, \quad i=1,\dots,N \,.
\end{equation}
\end{itemize}

Note the similarity to deploying the particle method for solving the Wasserstein gradient flow~\eqref{eqn:GF_wasserstein}. 
It is a standard technique to adopt the particle presentation:
\begin{equation} \label{eq:he1}
    \rho_t \approx \overline{\rho}_t = \frac{1}{N} \sum_{i=1}^N \delta_{X^i(t)} \,.
\end{equation}
When this form is substituted into~\eqref{eqn:GF_wasserstein}, we arrive at:
\begin{equation}\label{eqn:particle_GF}
\dot{X}^i = - \nabla_x \frac{\delta E}{\delta \rho}[\overline{\rho}_t](X^i) \,, \quad i = 1,\dots, N \,.
\end{equation}

\begin{remark}
The current paper focuses only on the convergence of the continuous-time dynamics. 
To make these approaches practical, we need to investigate the errors that arise in their numerical implementations. 
Our implementations make use of an adaptive solver (Section~\ref{sec:numerical_evidence}).
We do not attempt an error analysis that is customized to the specific form of these differential equations.
We next discuss the two sources of numerical error: time discretization and particle representation.
\begin{itemize}
\item Time discretization: When the time-step size $h$ is small, the discrete solution should capture the fast decay rate of the continuous solution.
In particular, the discretization error has to be compatible with the convergence rates of the PDE. 
This property is termed ``rate-matching discretization'' in~\citet{WiWiJo:2016variational} in the Euclidean setting. 
Various techniques have been proposed, among which the symplectic integrator~\citep{MuJo:2021optimization} appears to produce higher order accuracy. 
This observation is in line with the proposal presented in~\citet{AmGa:2008hamiltonian}, which does not  discuss details. 
It would be interesting to adapt these techniques to the setting of probability measure space. 
However, we note that the nonlinear geometry (compared to~\citet{WiWiJo:2016variational,BeJoWi:2018symplectic,MuJo:2021optimization} in Euclidean distance) makes the analysis significantly harder. 
\item Particle representation: It is widely known that Monte Carlo particle approximation suffers from the curse of dimensionality~\citep{SiPo:2018minimax,NiBe:2022minimax}, since the number of particles needed to represent the underlying distribution scales exponentially in the dimension of the problem. 
However, the ultimate task is to find the minimizer of the energy functional. 
This is a ``weak-sense" evaluation of the convergence, and there may still be a chance to achieve convergence without experiencing the curse of dimensionality. 
This intriguing possibility merits further investigation.
\end{itemize}
\end{remark}

\subsection{Numerical Results}\label{sec:numerical_evidence}

Here, we apply heavy-ball flow~\eqref{eqn:particle_hb}, Nesterov flow~\eqref{eqn:particle_nes}, and exponentially convergent variational acceleration flow~\eqref{eqn:particle_exp} to three tasks. 
Example 1 involves minimization of a  potential energy functional. 
Example 2 minimizes a KL divergence against a given target distribution, a problem from Bayesian sampling. 
Example 3 is the training of an infinitely-wide, single-layer neural network with ReLU activation.

In all these examples, the numerical integration over time is performed using the \textsc{Diffrax} library~\citep{Ki:2021on}. We use the Dormand-Prince 5/4 method with the default adaptive step size controller, setting both relative and absolute tolerances to $10^{-6}$.
The initial conditions of the particles are independently sampled from the standard Gaussian distribution: $X^i(0), V^i(0) \overset{\text{i.i.d.}}{\sim} \Nc(0,I_d)$. 

For strongly convex functions, the choice $a = 2 \sqrt{m}$ used in the analysis is too small to produce optimal computational performance. Thus in all examples,  heavy-ball flow is executed with $a=0.5$.

\medskip
\noindent
{\bf Example 1: Potential Energy.}
We consider potential energy
\[
E[\rho] = \Vc_\ell[\rho] = \int_{\Rb^d} V_\ell(x)\, \rmd \rho \,, \quad \ell=1,2,
\]
with two different forms of the potential functions: 
\begin{equation} \label{eq:cu1}
V_1(x) = \frac{1}{2} \langle x-b, A (x-b)\rangle \,, \qquad V_2(x) = h\log \left( \sum_{i=1}^M \exp\left( \frac{\langle w_i, x\rangle - q_i}{h} \right) \right) \,.
\end{equation}
For the potential $\Vc_1$, we set spatial dimension to be $d=500$, with  $A\in\Rb^{500\times500}$ a random symmetric 
positive definite matrix and $b$ is a random vector. 
The symmetric matrix is formed by setting $A = U^\top D U$ with $D$ being a diagonal matrix and $U$ being an orthogonal matrix.
The diagonal elements in $D$ are log-uniformly distributed between $10^{-5}$ and $1$, and $U$ is uniformly sampled from the Haar measure over the orthogonal group $\mathrm{O}(d)$. 
The vector $b$ is drawn from  $\Nc(0,100\cdot I_d)$.
This design ensures the objective functional $\Vc_1$ to be $m$-strongly convex with $m\approx10^{-5}$. For potential $\Vc_2$, we take  $d=200$ and choose $M=1000$ and $h=20$. 
Each $w_i\in\Rb^{200}$ is drawn from $\Nc(0,I_{200})$ and $q=(q_{i})\in\Rb^{1000}$ is from $\Nc(0,I_{1000})$.
Depending on the choice of $w_i$, $\Vc_2$ can be strongly convex, but its eigenvalues can be as small as one wishes.
We use $N=100$ particles for  both $\Vc_1$ and $\Vc_2$. 

In both examples, $E[\rho_t]$ is estimated empirically from \eqref{eq:he1}. To estimate the optimal $E_\ast$ for $V_2$, we run all four methods for a long time and designate $E_\ast$ the best value achieved over these four runs.

Figure~\ref{fig:convergence-potential} shows the decrease in optimality gap $E[\rho_t] - E_*$ over time for heavy-ball flow,  Nesterov flow, exponentially convergent variational acceleration flow, and Wasserstein gradient flow.
All three solvers demonstrate faster continuous-time convergence than the standard Wasserstein gradient flow, with exponentially convergent VAF being the fastest.

\begin{figure}[htbp]
  \centering
  \includegraphics[width=0.4\textwidth]{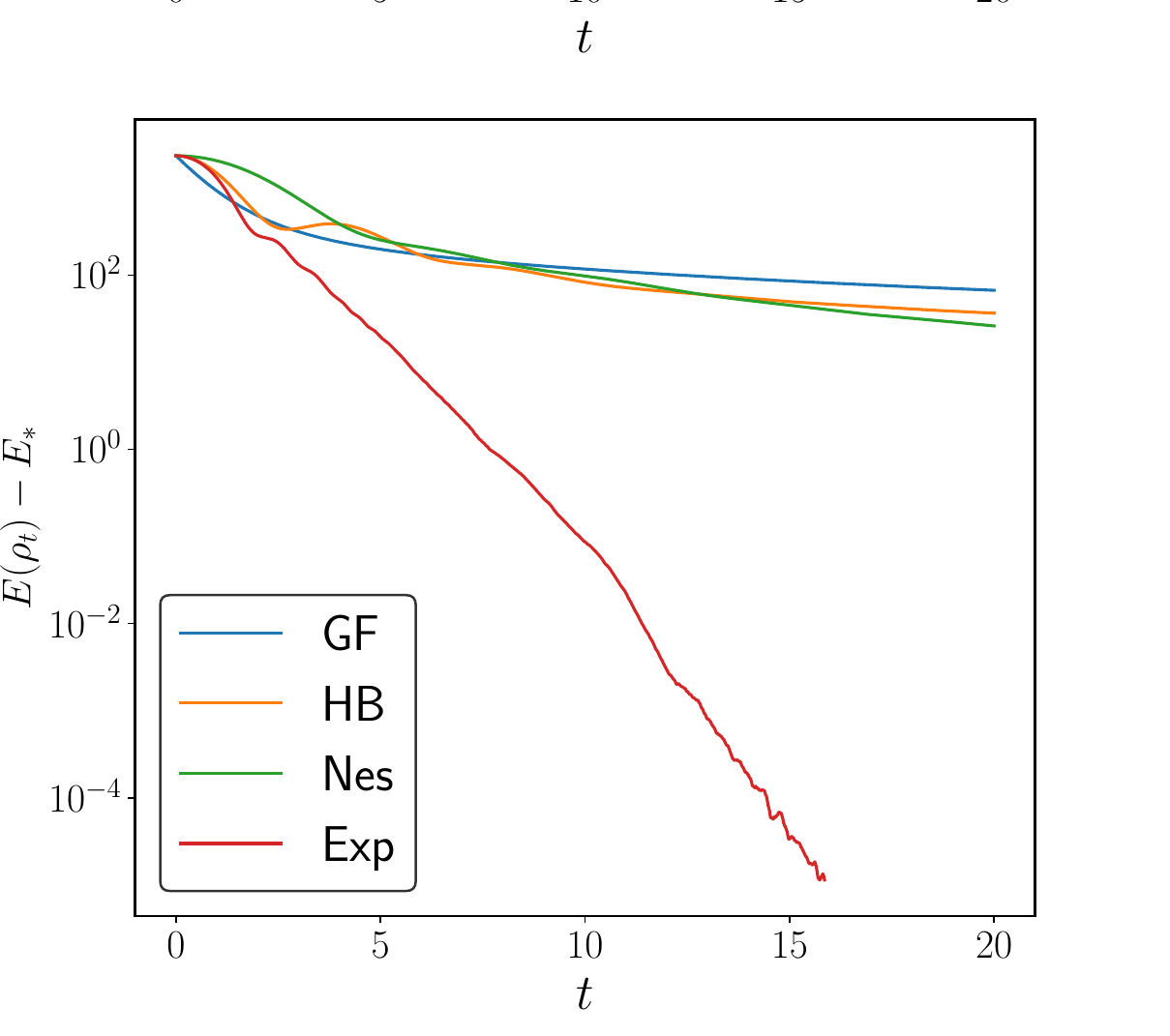}
  \includegraphics[width=0.4\textwidth]{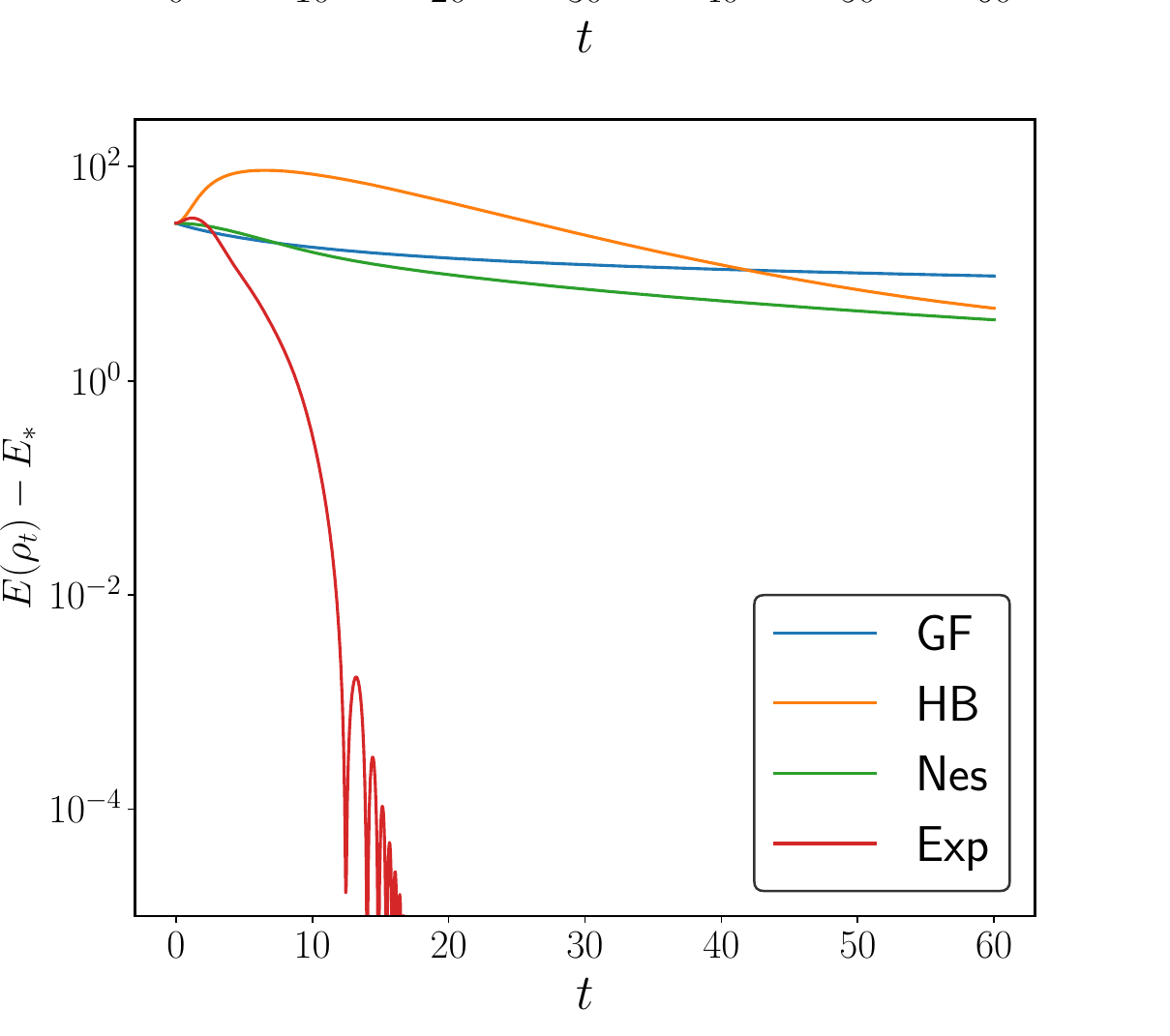} 
  \caption{Optimality gap vs time for Wasserstein gradient flow (WGF), Heavy-Ball flow (HB),  Nesterov flow (Nes) and exponentially convergent Hamiltonian flow (Exp), for  the functionals $\Vc_1$ (left) and $\Vc_2$ (right). 
  The functionals are evaluated at empirical measures; see \eqref{eq:he1}.
  }
  \label{fig:convergence-potential}
\end{figure}

In Figure~\ref{fig:cost-potential} we show the total number of steps for the four methods, for different tolerance of the optimality gap. 
As the convergence tolerance is tightened (moving toward the right of the plot), the number of steps required grows.
For any given problem, the total number of steps is an effective proxy for the actual computational cost, as it is proportional to the number of gradient evaluations.
By comparing the two plots, we see that the Nesterov and 
Hamiltonian flows outperform the Wasserstein gradient flow for small tolerance.

\begin{figure}[htbp]
  \centering
  \includegraphics[width=0.4\textwidth]{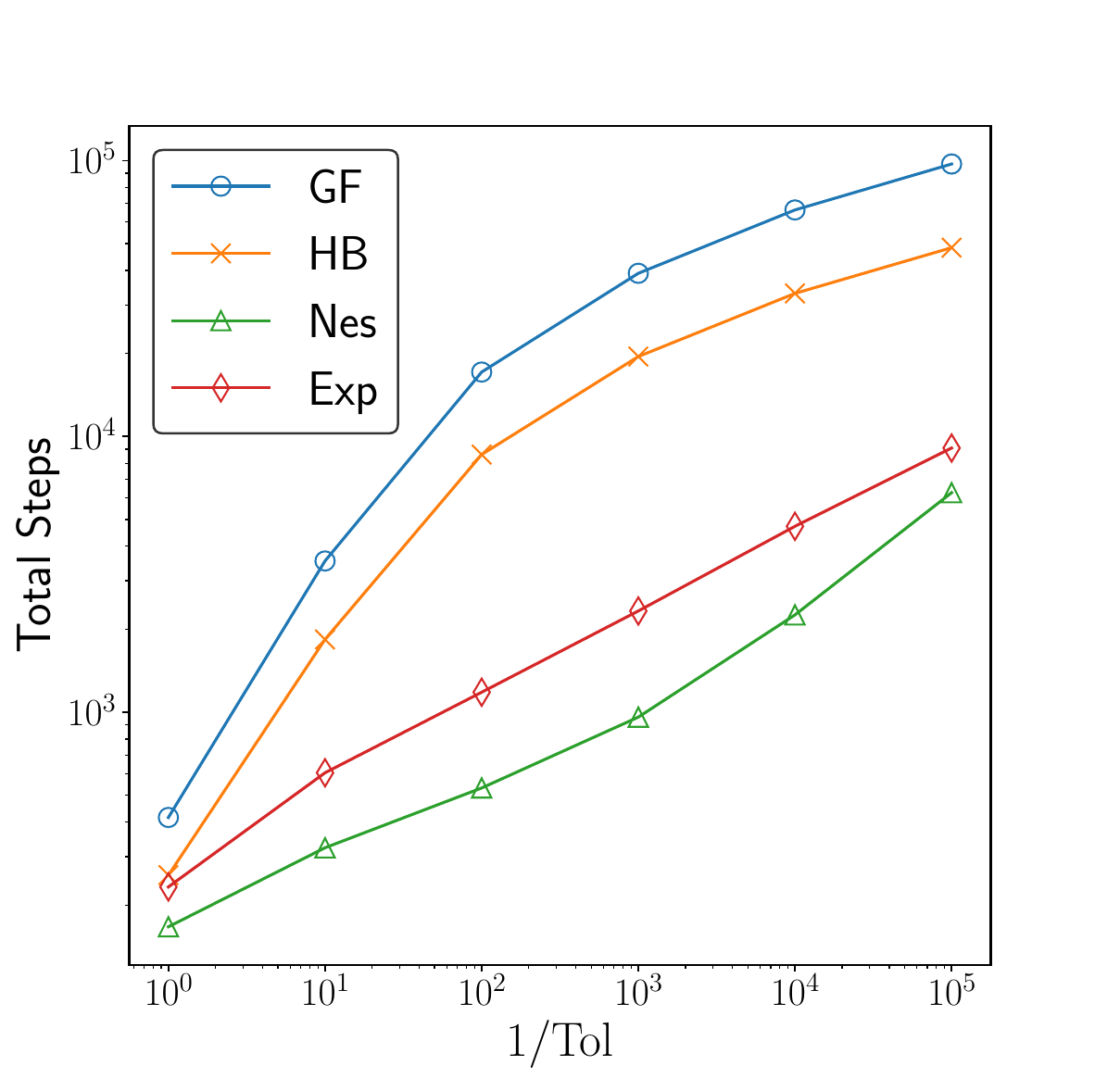}
  \includegraphics[width=0.4\textwidth]{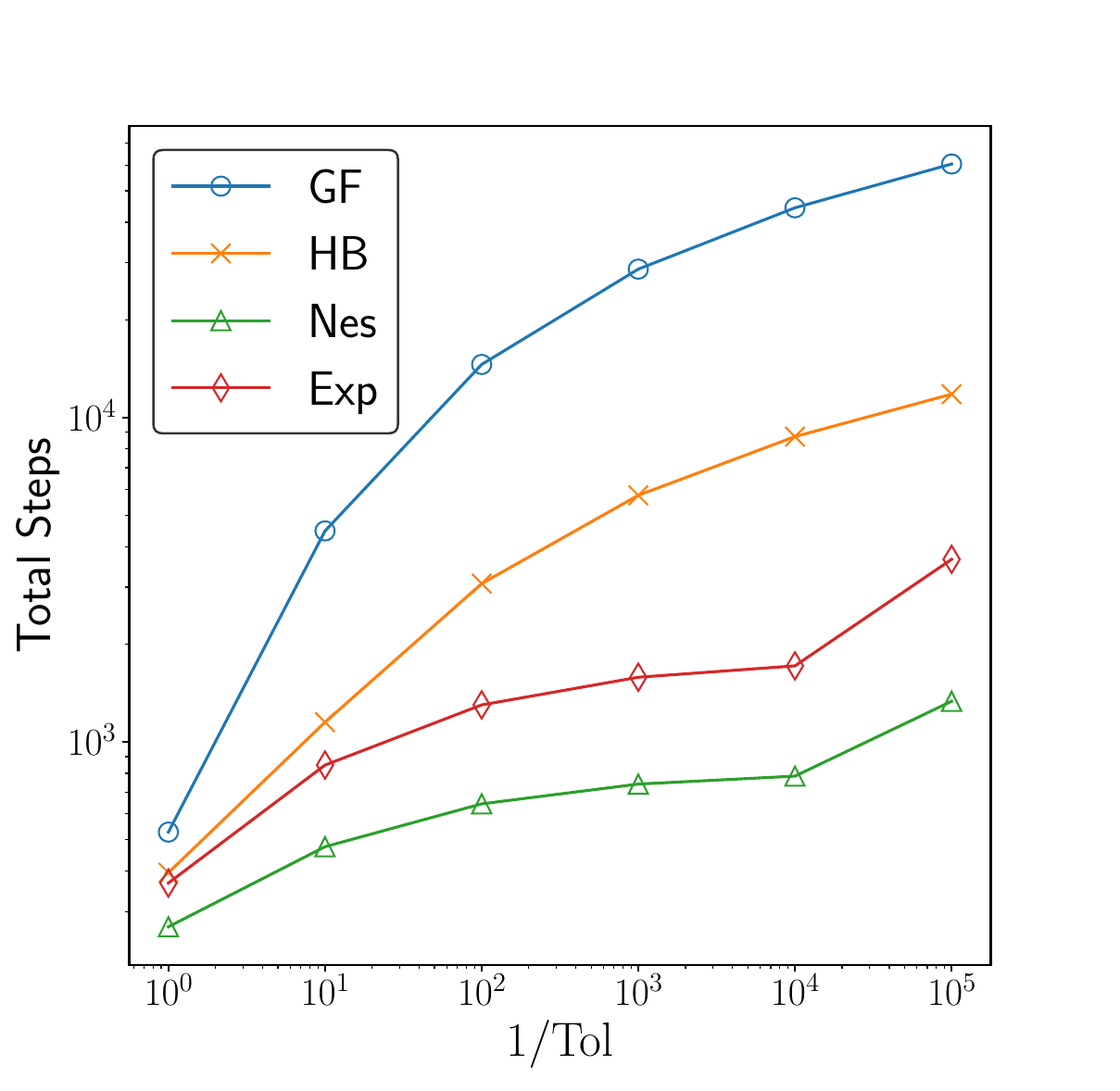}
  \caption{Total number of steps vs optimality gap (Tol) for Wasserstein gradient flow (WGF), Heavy-Ball flow (HB),  Nesterov flow (Nes) and exponentially convergent Hamiltonian flow (Exp), for the functionals $\Vc_1$ (left) and $\Vc_2$ (right). 
  The functionals are evaluated at empirical measures; see \eqref{eq:he1}.
  The total number of steps includes those accepted and rejected in the adaptive step size controller. 
}
  \label{fig:cost-potential}
\end{figure}

\medskip\noindent{\bf Example 2: Bayesian Sampling.} 
Next, we tackle the more challenging task of minimizing KL divergence between $\rho$ and a target distribution $\rho_\ast$, defined by
\[
E[\rho] = KL(\rho||\rho_\ast) = \int_{\Rb^d} \rho(x) \ln\frac{\rho(x)}{\rho_\ast(x)} \rmd x = \int_{\Rb^d}  \ln\rho(x) \, \rmd \rho + \int_{\Rb^d} g(x) \, \rmd \rho - \log C \,,
\]
where the target measure has density $\rho_\ast(x) = Ce^{-g(x)}$, with $C>0$ being the normalizing constant. 
With a slight abuse of notation, we do not distinguish a probability measure from its Lebesgue density. 
We cannot apply the particle approximation directly to the KL divergence because the empirical measure lacks a Lebesgue density. 
By analogy with the blob method for the Fokker-Planck equation~\citep{CaCrPa:2019blob}, we consider a regularized KL divergence
\begin{equation}
E^\eps[\rho] = KL^\eps(\rho||\rho_\ast) = \int_{\Rb^d} \ln K^\eps\ast\rho\, \rmd \rho + \int_{\Rb^d} g\, \rmd \rho - \log C \,,
\end{equation}
with the Gaussian convolution kernel $K^\eps(x) = (\frac{1}{2\pi\eps^2})^{d/2} e^{-|x|^2/2\eps^2}$ for some parameter $\eps>0$. This convolution allows $E^\eps$ being well-defined even for empirical measures. We note that $E^\eps$, like $E$, is lower-bounded and is convex assuming an $m$-strongly convex $g(x)$, with $m$ being sufficiently large, as outlined in~\cite{CaCrPa:2019blob}. 

We choose two different target measures $\rho_\ast$ by specifying the log-density $g$ in the same manner as the potential functions in \eqref{eq:cu1}, that is, 
\begin{equation}
g_1(x) = \frac{1}{2} \langle x-b, A (x-b)\rangle \,, \qquad
g_2(x) = h\log \left( \sum_{i=1}^M \exp\left( \frac{\langle w_i, x\rangle - q_i}{h} \right) \right) \,.
\end{equation}
For $g_1$, we take $d=20$ so that $A\in\Rb^{20\times20}$ is a random positive definite symmetric  
matrix, and $b$ is a random vector drawn from $\Nc(0,10 \cdot I_{20})$.
The symmetric matrix is formed by setting $A = U^\top D U$ with $D$ being a diagonal matrix and $U$ being an orthogonal matrix.
The diagonal elements in $D$ are log-uniformly distributed between $10^{-4}$ and $1$, and $U$ is uniformly sampled from the Haar measure over the orthogonal group $\mathrm{O}(d)$.
For $g_2$, we set $d=10$, $M=200$, and $h=10$. 
Each $w_i\in\Rb^{10}$ is drawn from $\Nc(0,I_{10})$, while $q=(q_{i})\in\Rb^{200}$ is drawn from $\Nc(0,I_{200})$. 
In the numerical results below, we use $N=1600$ particles. 
We choose $\epsilon=1$ (deferring  the issue of choosing $\epsilon$ more optimally to future work).

To estimate the optimal $E^\eps_\ast$, we run all four methods for a long time and designate $E^\epsilon_\ast$ the best value achieved over these four runs.

Figure~\ref{fig:convergence-KL} shows optimality gap as a function of $t$ for heavy-ball flow, Nesterov flow, exponentially convergent variational acceleration flow, and Wasserstein gradient flow. 
In both examples, the exponentially convergent Hamiltonian flow achieves the fastest convergence rate. 
It is noticeable that the Wasserstein gradient flow displays relatively lower errors during the initial stages, while the Hamiltonian flows exhibit slower decay due to oscillations. 
For larger $t$, the error of WGF saturates while the oscillations seen in the Hamiltonian flows taper off. 
Figure~\ref{fig:cost-KL} shows the total number of steps as a function of the optimality gap.
In both examples, the Nesterov and Hamiltonian flows outperforms the Wasserstein gradient flow for small tolerance.
The performance of the Hamiltonian flows could be potentially enhanced by mitigating the oscillation through the incorporation of restarting strategies~\citep{OdCa:2015adaptive,SuBoCa:2014differential}.

\begin{figure}[htbp]
  \centering
  \includegraphics[width=0.4\textwidth]{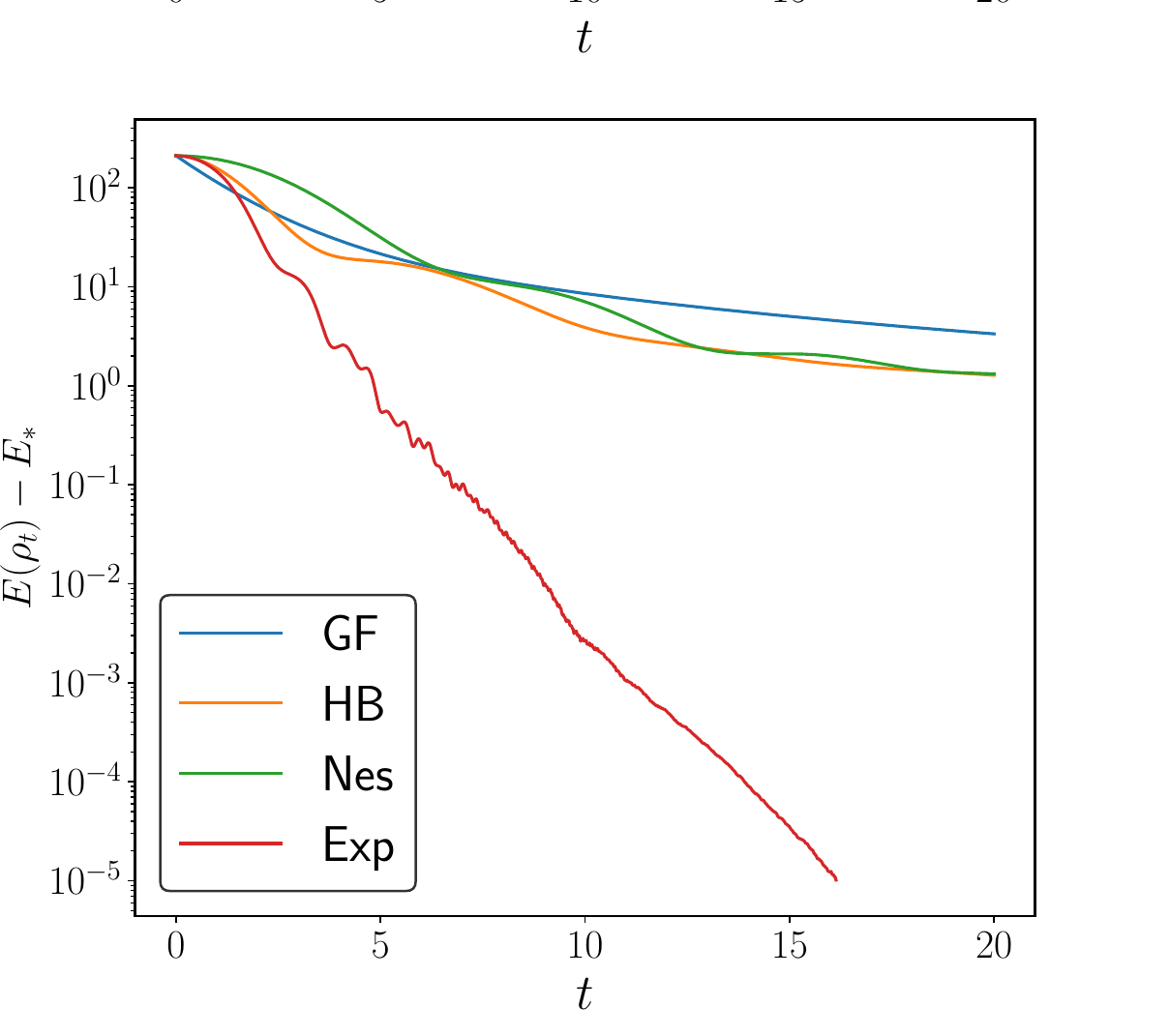}
  \includegraphics[width=0.4\textwidth]{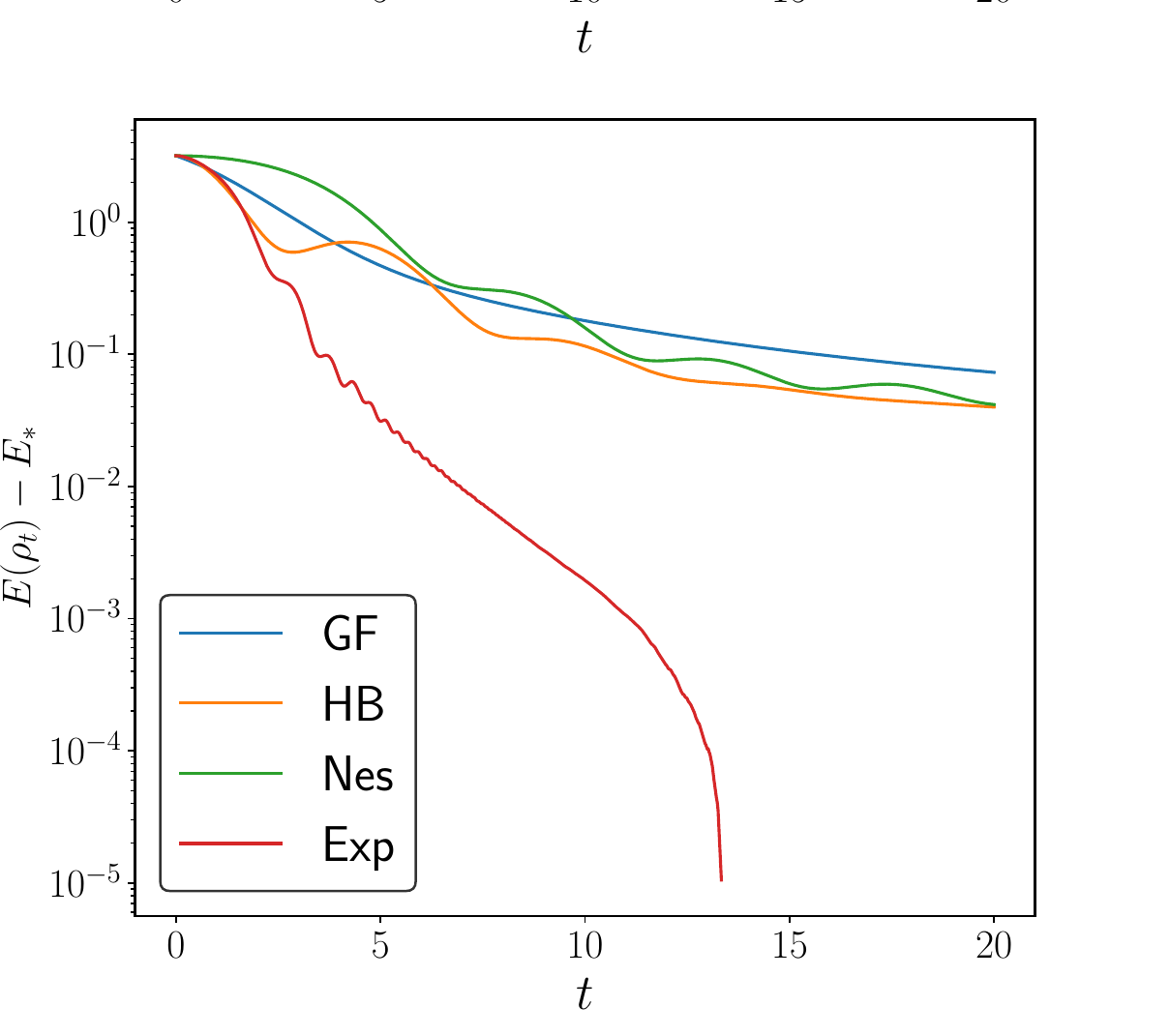}
  \caption{Optimality gap for minimization of regularized KL divergence with target $g_1$ (left) and $g_2$ (right) for four methods: Wasserstein gradient flow (WGF), Heavy-Ball flow (HB), Nesterov flow (Nes), and exponentially convergent Hamiltonian flow (Exp). The functionals are evaluated at empirical measures; see \eqref{eq:he1}.
  }
  \label{fig:convergence-KL}
\end{figure}

\begin{figure}[htbp]
  \centering
  \includegraphics[width=0.4\textwidth]{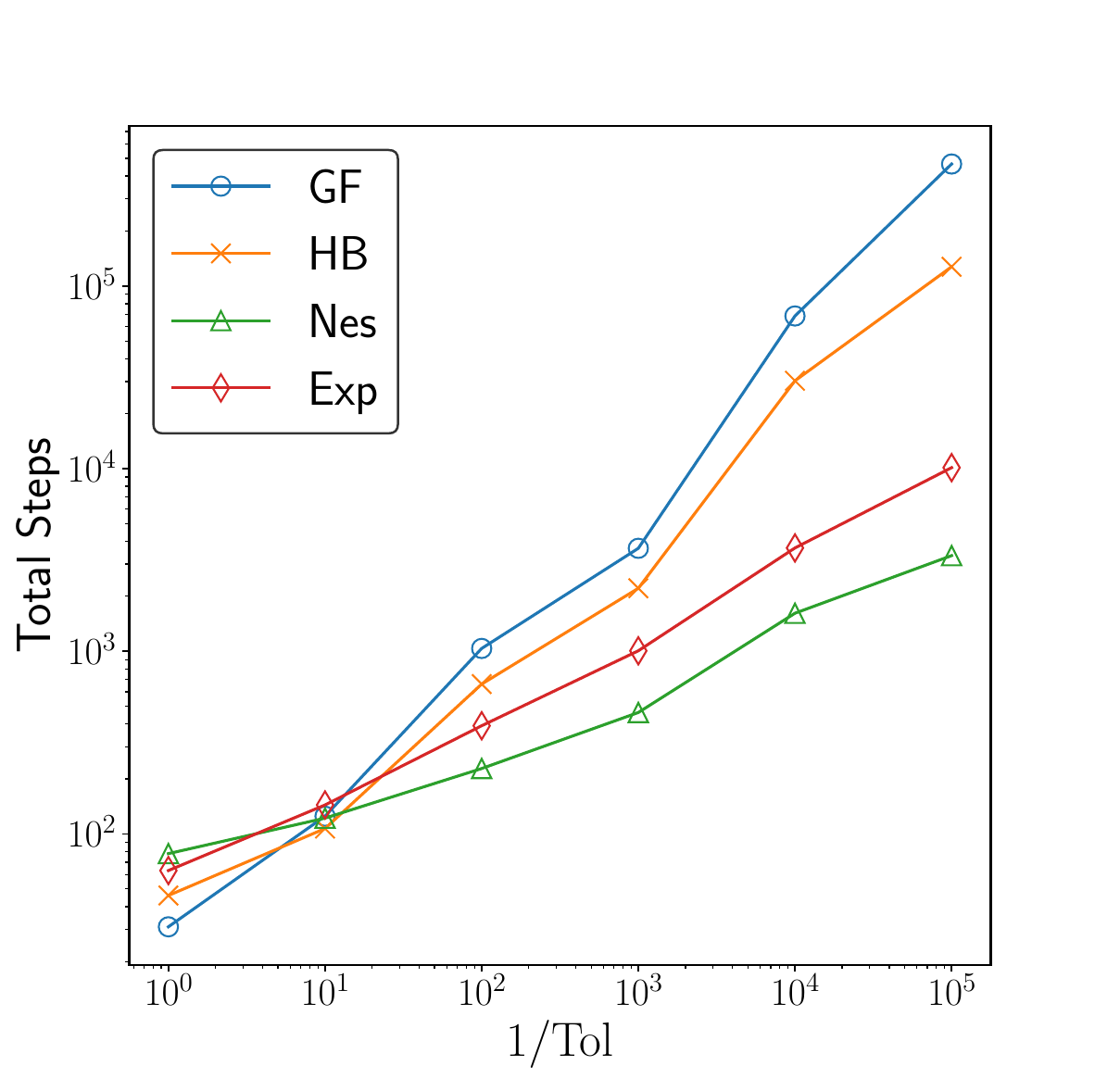}
  \includegraphics[width=0.4\textwidth]{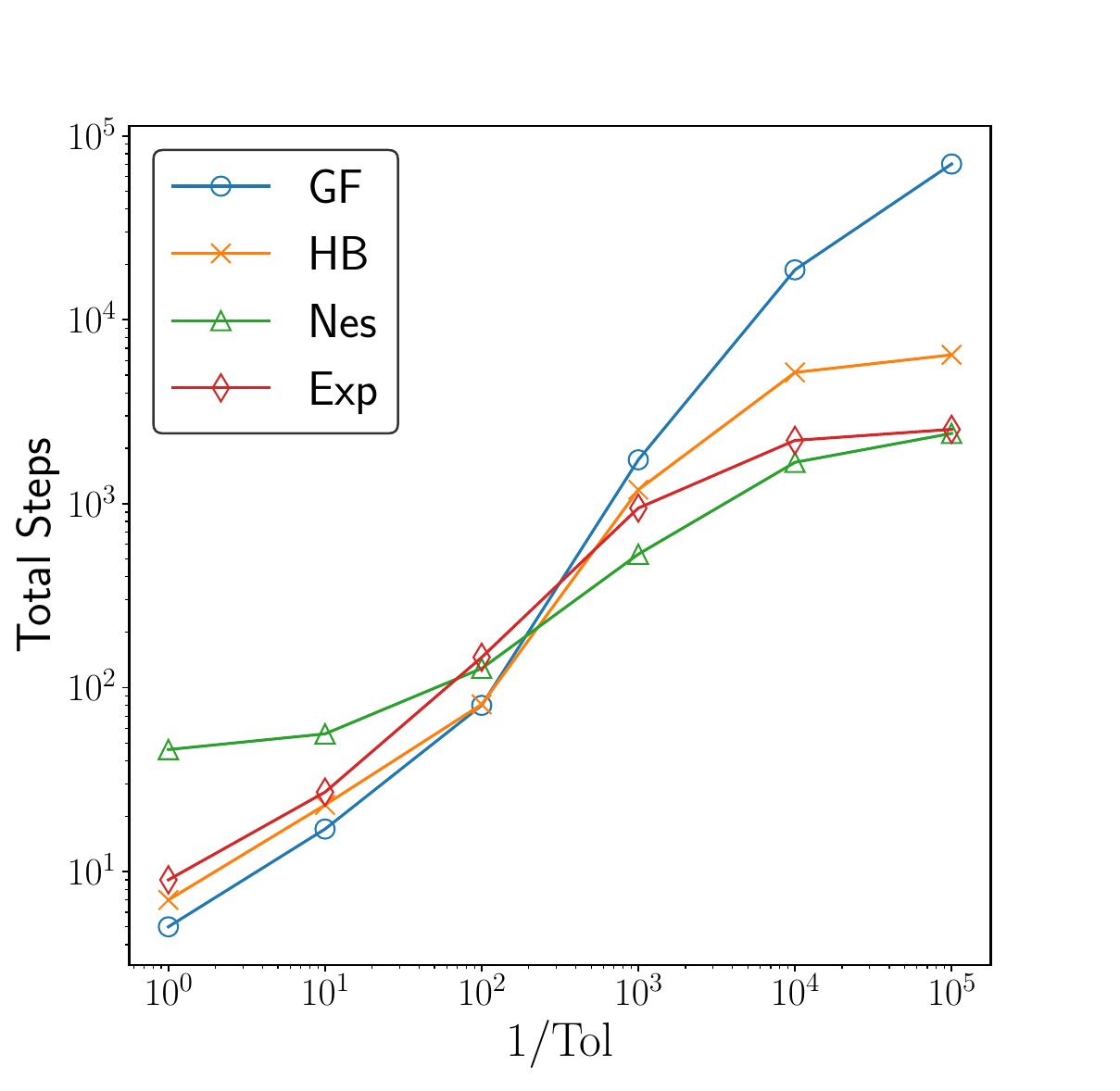}
  \caption{Total number of steps vs optimality gap (Tol) for minimization of regularized KL divergence with target $g_1$ (left) and $g_2$ (right) for four methods: Wasserstein gradient flow (WGF), Heavy-Ball flow (HB), Nesterov flow (Nes), and exponentially convergent Hamiltonian flow (Exp). The functionals are evaluated at empirical measures; see \eqref{eq:he1}.
  The total number of steps includes those accepted and rejected in the adaptive step size controller.
  }
  \label{fig:cost-KL}
\end{figure}

\medskip\noindent
{\bf Example 3: Neural network training.} 
Our final example is related to the training of infinitely wide neural networks~\citep{ChBa:2018global,MeMoNg:2018mean,SiSp:2020mean,DiChLiWr:2021global,DiChLiWr:2022overparameterization}, where we have 
\begin{equation}\label{eqn:energy_NN}
    E[\rho] = \frac{1}{2}\int_{\Rb^d} |f(x)-g(x,\rho)|^2\, \rmd \pi(x),
\end{equation}
where $\pi$ is a given distribution over the sampled data and $f:\Rb^d\to\Rb$ is the target function. We take the function $g$ to be a two-layer neural network: for every $x\in\Rb^d$ and $\rho\in\Pc(\Rb^{d+3})$
\[
    g(x,\rho) \coloneqq \int_{\Rb^{d+3}} V(x,z)\,\rmd \rho(z),\quad \text{with} \quad V(x,(\alpha,\beta,w,b)) = \alpha\,\sigma(w\cdot x+b)+\beta\,,
\]
with $\sigma$ being the ReLU function, which is positively $1$-homogeneous, and $z=(\alpha,\beta,w,b)\in\Rb\!\times\Rb\!\times\Rb^d\!\times\Rb$. The functional $E[\rho]$ cannot be shown to be globally geodesic convex, but is locally geodesic convex, see Appendix~\ref{app:nn_training}. We nevertheless tested the training with four methods (Wasserstein gradient flow (WGF), Heavy-Ball flow (HB), Nesterov flow (Nes), and exponentially convergent Hamiltonian flow (Exp)). 

We set the spatial dimension to be $d=1$ with the target function being $f(x) = \sin(\pi x)$. We choose the data distribution to be the uniform distribution over $[-1,1]$ and $500$ data points are sampled to evaluate the integration in $\pi$. In the numerical results below, we use $N=200$ particles (neurons). The numerical results are presented in Figure~\ref{fig:NN}.
Note that the Nesterov and Hamiltonian flows outperform the Wasserstein gradient flow for small tolerances.

\begin{figure}[htbp]
  \centering
  \includegraphics[width=0.3\textwidth]{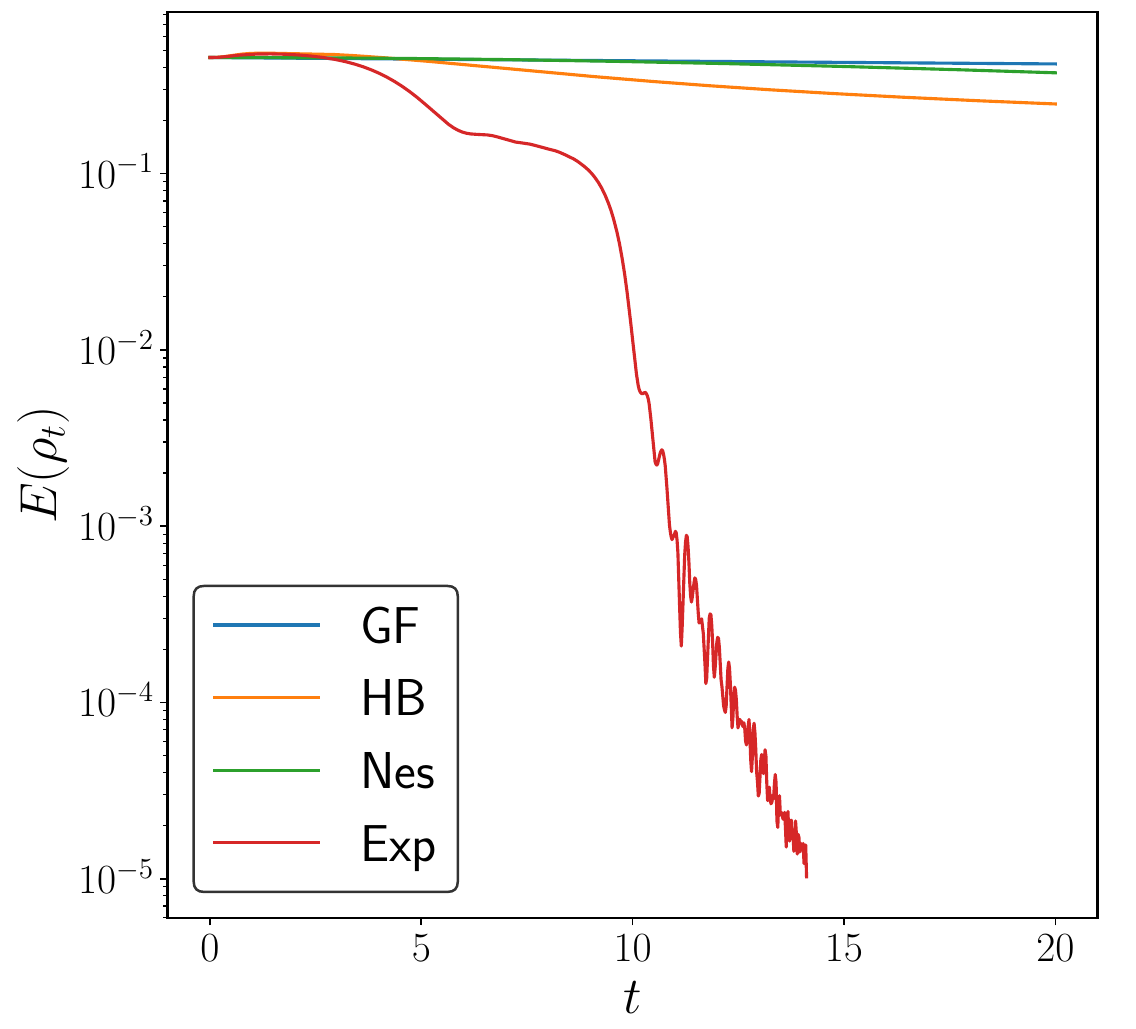}
  \includegraphics[width=0.3\textwidth]{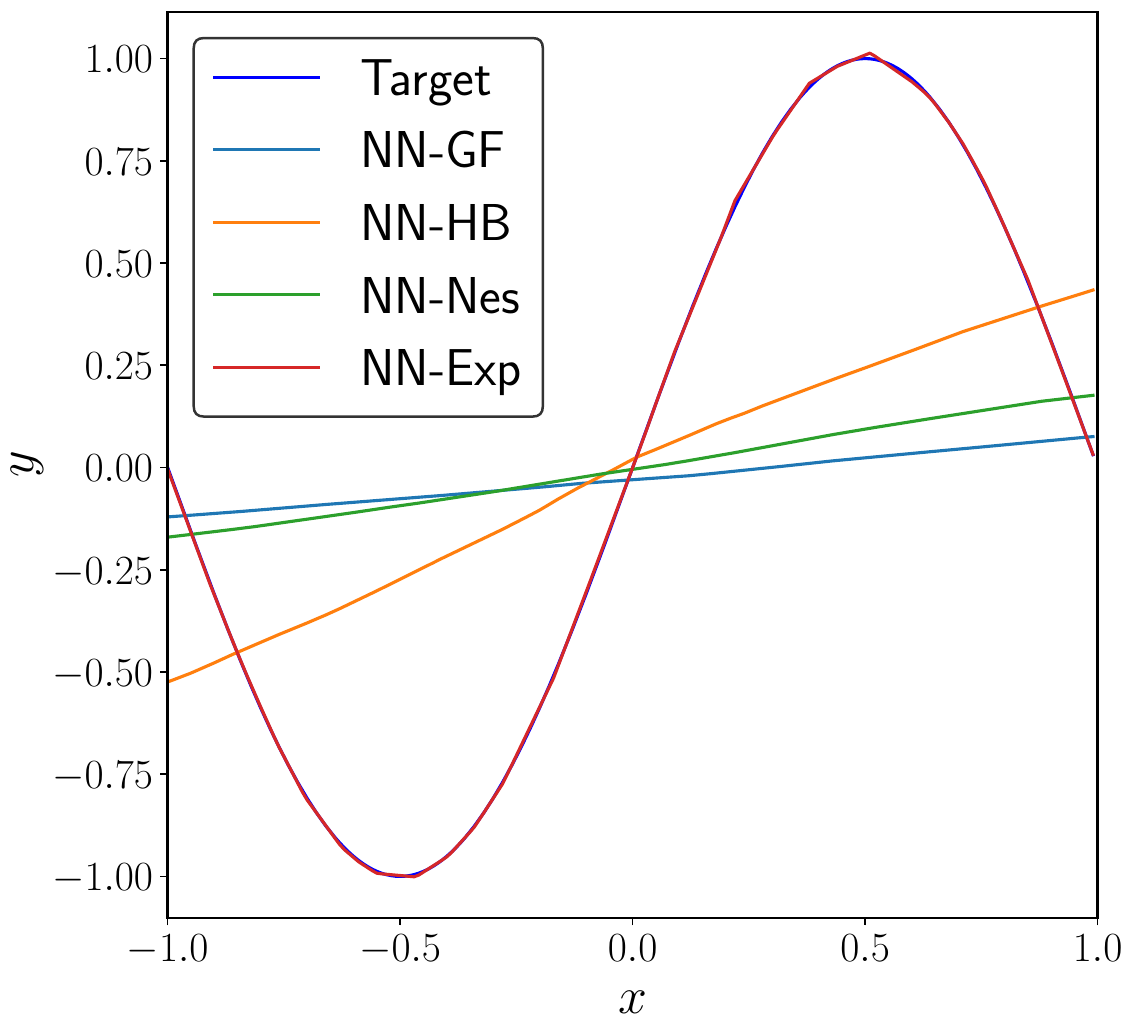}
  \includegraphics[width=0.3\textwidth]{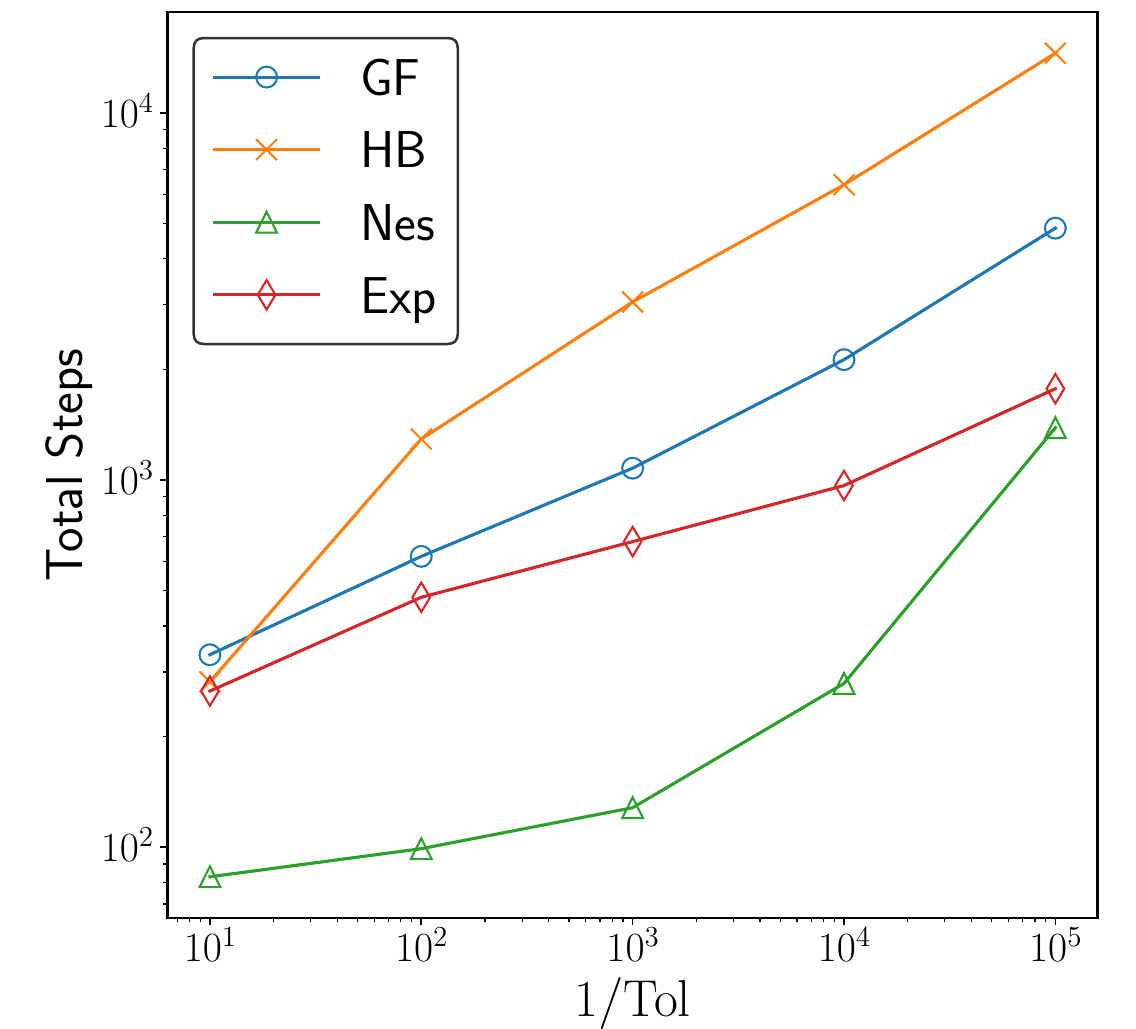}
  \caption{\textbf{Left:} Mean square errors for neural network training with target $f(x) = \sin(\pi x)$ for four methods: Wasserstein gradient flow (WGF), Heavy-Ball flow (HB), Nesterov flow (Nes), and exponentially convergent Hamiltonian flow (Exp). \textbf{Middle:} The target $f(x) = \sin(\pi x)$ and its neural network approximations obtained by running the four methods for $T=14$. \textbf{Right:} The number of total steps as a function of the mean square error (Tol). The number of total steps includes those accepted and rejected in the adaptive step size controller.
  }
  \label{fig:NN}
\end{figure}

\acks{Q.L. acknowledges support from Vilas Early Career award and ONR-N00014-21-1-214. The research of S.C. and Q.L. is supported in part by NSF-CAREER-1750488 and Office of the Vice Chancellor for Research and Graduate Education at the University of Wisconsin Madison with funding from the Wisconsin Alumni Research Foundation. The work of S.C., Q.L., and S.W. is supported in part by NSF via grant DMS-2023239. S.W. also acknowledges support from NSF Award CCF-2224213 and  AFOSR FA9550-21-1-0084. O.T.\ acknowledges support from NWO Vidi grant 016.Vidi.189.102 on {\em Dynamical-Variational Transport Costs and Application to Variational Evolutions}, and NWO grants OCENW.M.21.012 and NGF.1582.22.009.}

\newpage


\appendix

\section{Examples of Hamiltonian flows}\label{app:examples}

We present several further  examples of the Hamiltonian PDE.

\medskip\noindent{\bf Example 4: Kalman-Hamiltonian flow~\citep{GaHoLiSt:2020interacting,WaLi:2022accelerated,LiStWa:2022second}.}
Consider the Hamiltonian with weighted kinetic energy
\begin{equation}
    H_t[\mu] = e^{-a t } \frac{1}{2}\int_{\Rb^{2d}} \langle v, C^\lambda[\mu^X] v\rangle\, \rmd \mu + e^{a t } E[\mu^X]  \,,
\end{equation}
where $a>0$ and $C^\lambda[\nu]\in\Rb^{d\times d}$ is the covariance matrix defined by
\begin{equation}
    C^\lambda[\nu] = \int_{\Rb^d} (x-\Eb_{\nu}[X])\otimes(x-\Eb_{\nu}[X]) \rmd\nu + \lambda I, \qquad \lambda\ge 0.
\end{equation}
The Kalman-Hamiltonian flow is then
\begin{equation}
    \partial_t \mu_t + \nabla_x\cdot\left(\mu_t e^{-a t} C^\lambda[\mu_t] v \right) 
    -\nabla_v\cdot\left(\mu_t\left(e^{-a t}\Eb_{\mu_t^V}[V V^\top] (x-\Eb_{\mu_t^X}[X]) + e^{a t }\nabla_x\frac{\delta E}{\delta \rho}[\mu_t^X]\right)\right) = 0 \,.
\end{equation}
Particle dynamics are defined by
\begin{equation}
\begin{dcases}
    \dot{X}_t & = e^{-a t} C^\lambda[\mu_t^X] V_t \\
    \dot{V}_t & = -e^{-a t}\Eb_{\mu^V}[V_t V_t^\top] (X_t-\Eb_{\mu^X}[X_t]) -e^{a t }\nabla_x\frac{\delta E}{\delta \rho}[\mu_t^X](X_t).
\end{dcases}
\end{equation}

\medskip\noindent{\bf Example 5: Stein-Hamiltonian flow~\citep{DuNuSz:2019geometry,Li:2017stein,WaLi:2022accelerated}.}
Consider the Hamiltonian with a kernel-weighted kinetic energy
\begin{equation}
    H_t(\mu) =  e^{-a t }\int_{\Rb^{4d}} \frac{1}{2}  w^\top K(x,y) v \, \rmd \mu(x,v) \, \rmd \mu(y,w) + e^{a t } E[\mu^X]  \,,
\end{equation}
where $K(x,y)\in\Rb^{d\times d}$ is a symmetric positive kernel function. The Stein-Hamiltonian flow is then
\begin{equation}
\begin{aligned}
    \partial_t \mu_t &+ \nabla_x\cdot\left(\mu_t e^{-a t} \int_{\Rb^{2d}} K(x,y) w \rmd\rho_t(y,w) \right) \\
    & \quad - \nabla_v\cdot\left(\mu_t\left(e^{-a t} \int_{\Rb^{2d}} \nabla_x [v^\top K(x,y) w ]
    \rmd \mu_t(y,w) + e^{a t}\nabla_x\frac{\delta E}{\delta \rho}[\mu_t^X]\right)\right) = 0 \,.
\end{aligned}
\end{equation}
Particle dynamics are defined by
\begin{equation}
\begin{dcases}
    \dot{X}_t & = e^{-a t} \int_{\Rb^{2d}} K(X_t,y) w \rmd\mu_t(y,w)\\
    \dot{V}_t & = -e^{-a t} \int_{\Rb^{2d}} \nabla_x [V_t^\top K(X_t,y) w ]
    \rmd \mu_t(y,w) -e^{a t }\nabla_x\frac{\delta E}{\delta \rho}[\mu_t^X](X_t).
\end{dcases}
\end{equation}

\medskip\noindent{\bf Example 6: Bregman-Hamiltonian flow~\citep{WiWiJo:2016variational,WiReJo:2016lyapunov}.}
The Bregman-Hamiltonian is defined by
\begin{equation}\label{eqn:bregman_hamiltonian}
    H_t(\rho) = e^{\alpha_t + \gamma_t } \left( \int_{\Rb^{2d}} D_{\psi^\ast}(\nabla \psi(x)+e^{-\gamma_t}v, \nabla \psi(x)) \rmd\mu(x,v) + e^{\beta_t } E[\mu^X] \right) \,,
\end{equation}
where $\psi:\Rb^d\to\Rb$ is a convex function of Legendre type~\citep{Ro:1997convex,BoLe:2006convex}, and $\psi^\ast:\Rb^d\to\Rb$ denotes the convex conjugate of $\psi$. 
(The Bregman divergence of a convex function $\psi$ is defined by
    $D_\psi(y,x) \coloneqq \psi(y) - \psi(x) - \langle \nabla \psi(x) , y-x \rangle$,
where $\langle \cdot , \cdot \rangle$ is the Euclidean inner product on $\Rb^d$.)
The Bregman-Hamiltonian flow is then
\begin{equation}\label{eqn:breg-hamil_flow_origin-form}
\begin{aligned}
    \partial_t \mu_t &+ \nabla_x\cdot\left(\mu_t e^{\alpha_t}\left( \nabla \psi^\ast(\nabla \psi(x)+e^{-\gamma_t}v)-x \right)\right) \\
    &-\nabla_v\cdot\left(\mu_t e^{\alpha_t +\gamma_t }\biggl[ \nabla^2\psi(x) \left( \nabla \psi^\ast\left( \nabla \psi(x)+e^{-\gamma_t}v \right) -x \right) - e^{-\gamma_t} v + e^{\beta_t}\nabla_x\frac{\delta E}{\delta \rho}[\mu_t^X] \biggr] \right) = 0.
\end{aligned}
\end{equation}
Particle dynamics for this flow are defined by
\begin{equation}
\begin{aligned}
    \dot{X}_t & = e^{\alpha_t}\left( \nabla \psi^\ast(\nabla \psi(X_t)+e^{-\gamma_t}V_t)-X_t \right) \\
    \dot{V}_t & = -e^{\alpha_t +\gamma_t } \biggl[\nabla^2\psi(X_t) \left( \nabla \psi^\ast\left( \nabla \psi(X_t)+e^{-\gamma_t}V_t \right) -X_t \right) + e^{-\gamma_t} V_t - e^{\beta_t }\nabla_x\frac{\delta E}{\delta \rho}[\mu_t^X](X_t)\biggr].
\end{aligned}
\end{equation}
We can define the mirror transform $M_t(x,v) = (x,z) = (x,\nabla \psi(x) + e^{-\gamma_t} v)$ and the pushforward measure $\nu_t = (M_t)_\sharp \mu_t \in \Pc_2(\Rb^{2d})$. Then $\nu_t^X = \mu_t^X$ and $\nu_t$ solves the equation
\begin{equation}\label{eqn:breg-hamil_flow_primal-dual-form}
    \partial_t \nu_t + \nabla_x\cdot\left(\nu_t e^{\alpha_t}\left( \nabla \psi^\ast(z)-x \right)\right) -\nabla_z\cdot\left(\nu_t e^{\alpha_t +\beta_t }\nabla_x\frac{\delta E}{\delta \rho}[\nu_t^X] \right) = 0 \,,
\end{equation}
where we have used $\dot{\gamma}_t = e^{\alpha_t}$ in the derivation. 
Under these transformations, the associated particle dynamics becomes
\begin{equation}
\begin{aligned}
    \dot{X}_t & = e^{\alpha_t}\left( \nabla \psi^\ast(Z_t)-X_t \right) \\
    \dot{Z}_t & = -e^{\alpha_t +\beta_t }\nabla_x\frac{\delta E}{\delta \rho}[\nu_t^X](X_t).
\end{aligned}
\end{equation}
By choosing $\alpha_t = \log(r/t)$, $\beta_t=2\log(t/r)$ with $r>0$, this method generalizes the accelerated mirror descent method~\citep{KrBaBa:2015accelerated}.

\section{Rigorous statements and proofs for two properties in Section~\ref{sec:heavy-ball}}\label{app:derivatives_proof}

Here we present rigorous statements regarding the time derivatives of the Wasserstein distance. The following theorem characterizes the first-order derivative of the Wasserstein distance.
\begin{theorem}[Theorem 8.4.7 and Proposition 8.5.4 in~\cite{AmGiSa:2005gradient}]\label{thm:first_derivative}
Let $\sigma$ be a probability measure in $\Pc_2(\Rb^d)$ and $\nu \in C([0,\infty), \Pc_2(\Rb^d))$ be a solution to the continuity equation
\begin{equation}
    \partial_t \nu_t + \nabla \cdot (\nu_t \xi_t) = 0 \quad \text{ in distribution}\,,
\end{equation}
for locally Lipschitz vector fields $\xi$ satisfying
\begin{equation}
    \int_0^\infty \int_{\Rb^d} \|\xi_t\|^2 \rmd \nu_t \rmd t < \infty \,,
\end{equation}
then $\nu\in\AC([0,\infty),\Pc_2(\Rb^d))$ and for almost every $t\in(0,\infty)$, we have
\begin{equation}
\frac{1}{2} \frac{\rmd}{\rmd t} W_2^2(\nu_t,\sigma)
= \int_{\Rb^{2d}} \left\langle x-y, \xi_t(x) \right\rangle \,\rmd \gamma_t(x,y) \,,
\end{equation}
where $\gamma_t\in\Gamma_\rmo(\nu_t,\sigma)$.
\end{theorem}

Before proceeding to characterize the second-order derivative of the Wasserstein distance, it is necessary to define the disintegration of the phase space probability measure $\mu\in\mathcal{P}(\Rb^d\times\Rb^d)$ with respect to its $x$-marginal.
\begin{theorem}[Theorem 5.3.1 in~\citet{AmGiSa:2005gradient}]
    Let $\mu\in\mathcal{P}(\Rb^d\times\Rb^d)$, and denote by $\mu^X\in\mathcal{P}(\Rb^{d})$ the $x$-marginal distribution of $\mu$. Then there exists a $\mu^X$-a.e. uniquely determined family of Borel probability measure $\{\mu_{x}\}_{x\in\Rb^d}\in\mathcal{P}(\Rb^d)$ such that
    \[
    \int_{\Rb^{2d}} f(x,v)\, \rmd\mu(x,v) = \int_{\Rb^d} \biggl( \int_{\Rb^d} f(x,v)\, \rmd \mu_{x}(v) \biggr) \rmd \mu^X(x) \,,
    \]
    for every Borel map $f:\Rb^d\times\Rb^d\to[0,+\infty]$.
\end{theorem}

The second-order derivative of the $W_2$-distance can be computed by the following theorem.
\begin{theorem}[Modification of Theorem 1 in~\cite{CaChTs:2019convergence}]\label{thm:second_derivative}
Let $\sigma\in \Pc_2(\Rb^d)$ and $\mu \in AC([0,\infty), \Pc_2(\Rb^d\times\Rb^d))$ be a solution to the Hamiltonian flow
\begin{equation}
\partial_t \mu_t  + \nabla_x \cdot\left(\mu_t F(t,v) \right)
        -\nabla_v \cdot\left(\mu_t G(t,\mu_t,x) \right) = 0 \,,\quad \forall t\in[0,T]
\end{equation}
with locally-in-$t$ and globally-in-$(x,v)$ Lipschitz vector fields $(t,x,v)\mapsto F(t,v),\, G(t,\mu_t,x)$ satisfying
\begin{equation}\label{eqn:assumption_regularity}
t\mapsto \, \left\|F(t,v)\right\|_{L^2(\mu_t)} \,, \left\|\partial_t F(t,v)\right\|_{L^2(\mu_t)}\,,
\left\|\nabla_v F(t,v)\right\|_{L^2(\mu_t)}\,,
\left\|G(t,\mu_t,x)\right\|_{L^2(\mu_t)} \in C([0,\infty))\cap L^2([0,\infty)) \,.
\end{equation}
then for any $T>0$, the following inequality holds:
\begin{equation}\label{eqn:appendix_second_order_wasserstein}
\begin{aligned}
\frac{1}{2} \frac{\rmd^+}{\rmd t} \frac{\rmd}{\rmd t} W_2^2(\mu_t^X,\sigma)
\leq &\int_{\Rb^{2d}} \|F(t,v)\|^2 \rmd \mu_t(x,v) \\
&+ \int_{\Rb^{3d}} \left\langle x-y, \partial_t F(t,v)  - \nabla_v F(t,v) G(t,\mu_t,x) \right\rangle \, \rmd \mu_{t,x}(v) \rmd \gamma_t(x,y) \,,
\end{aligned}
\end{equation}
where $\rmd^+/\rmd t$ denotes the upper derivative in almost every $t>0$ and the Jacobian $\nabla_v F = (\partial_{v^j} F^i)_{ij}$. Here $\mu_{t,x}$ is the disintegration of $\mu_t$ with respect to its $x$-marginal $\mu_t^X$.
\end{theorem}
\begin{proof}
We start by defining the following flow: For fixed $t\in(0,\infty)$, let $\Phi_\tau=(\Phi_\tau^X,\Phi_\tau^V)$ satisfy the following equations:
\begin{equation}\label{eqn:flow_xv}
\begin{aligned}
\partial_\tau \Phi^X_\tau(x,v) &= F(t+\tau,\Phi^V_\tau(x,v)) \,,\\
\partial_\tau \Phi^V_\tau(x,v) &= - G(t+\tau,\mu_{t+\tau},\Phi^X_\tau(x,v)) \,,
\end{aligned} \qquad (\Phi_0^X,\Phi_0^V)=(x,v) \,, \text{ for } \mu_t\text{-a.e.}\,(x,v)\,.
\end{equation}
These formulas define the corresponding Lipschitz flow for $\tau\in(-t,\infty)$. Set
\begin{equation}
\mu_{t\pm h} = (\Phi_{\pm h})_\sharp \mu_t \,.
\end{equation}
By defining
$J(t,\mu,x,v) \coloneqq \partial_t F(t,v)-\nabla_v F(t,v) G(t,\mu,x)$,
one can compute
\begin{equation}\label{eqn:pushforward_second}
\partial_\tau^2 \Phi_\tau^X(x,v) =  J(t+\tau,\mu_{t+\tau},\Phi_\tau(x,v)) \,.
\end{equation}

Let $\gamma_t\in\Gamma_\rmo(\mu_t^X,\sigma)$ be an optimal transport plan between $\mu_t^X$ and $\sigma\in\Pc_2(\Rb^d)$. 
To prove~\eqref{eqn:appendix_second_order_wasserstein}, we use an approximation argument based on finite differences. To this end, we evolve $\gamma_t$ in time, using the map $\Phi_\tau$, such that it remains a transport plan between the $x$-marginal $\mu_{t+\tau}^X$ and the measure $\sigma$. 
However, to execute the flow~\eqref{eqn:flow_xv}, both the initial coordinate $x$ and the velocity $v$ are required, whereas the velocity information is absent in $\gamma_t$. 
To address this issue, we define an extended coupling plan over $\mathbb{R}^{3d}$ by appending the velocity information into $\gamma_t$.
Specifically, we define the extended coupling plan $\gamma_t^\rmE$ as the probability measure over $\mathcal{P}(\Rb^{3d})$ that satisfies the equation:
\[
\int \phi(x,v,y)\, \rmd \gamma_t^\rmE(x,v,y) = \int \phi(x,v,y)\, \rmd\mu_{t,x}(v)\rmd \gamma_t(x,y) \,, \quad \forall\phi\in C_\mathrm{b}(\Rb^{3d}) \,.
\]
To evolve the coupling plan $\gamma_t^\rmE$, we define the map $Q_\tau:\Rb^{3d}\to\Rb^{2d}$ as follows:
\[
    Q_\tau(x,v,y) = (\Phi_\tau^X(x,v),y) \,, \quad \tau\in(-t,\infty) \,.
\]
The evolved transport plan is then defined as the pushforward measure: $\gamma_t^\tau\coloneqq(Q_\tau)_\sharp \gamma_t^\rmE\in\mathcal{P}(\Rb^{2d})$.

We observe that $\gamma_t^\tau$ defines a transport plan between $\mu_{t+\tau}^X$ and $\sigma$.
Indeed, for any $\phi^X\in C_\rmb(\Rb^d)$, we integrate it against the plan $\gamma_t^\tau$ and obtain
\[
\begin{aligned}
    \int \phi^X(x)\, \rmd \gamma^\tau_t(x,v,y) &= \int \phi^X(\Phi_\tau^X(x,v))\, \rmd\mu_{t,x}(v) \rmd \gamma_t(x,y) \\
    &= \int \phi^X(\Phi_\tau^X(x,v))\, \rmd\mu_{t}(x,v) \\
    &= \int \phi^X(x)\, \rmd\mu^X_{t+\tau}(x) \,,
\end{aligned}
\]
where in the first equality, we use the definition of $\gamma_t^\tau$, and in the second equality we employ the definition of the transport plan $\gamma_{t}$ and the disintegration $\mu_{t,x}$. Similarly, for any $\phi^Y\in C_\rmb(\Rb^d)$, we integrate it against the plan $\gamma_t^\tau$ and obtain
\[
    \int \phi^Y(y)\, \rmd \gamma^\tau_t(x,v,y) = \int \phi^Y(y)\, \rmd\mu_{t,x}(v) \rmd \gamma_t(x,y) = \int \phi^Y(y)\, \rmd\sigma(y) \,.
\]
Thus, $\gamma_t^\tau$ is a transport plan between $\mu^X_{t+\tau}$ and $\sigma$ for any $\tau\in(-t,\infty)$.

The rest of the proof follows the same strategy as that of \citet[Theorem 5.3.1]{CaChTs:2019convergence}.
We outline it here for completeness.

For some fixed $t\in(0,\infty)$ and $h\in(0,t)$, consider the finite difference
\begin{equation}
\Delta_h \Kc(\mu_t^X,\sigma)
\coloneqq (D_{h/2}^2W_2^2)(\mu_t^X,\sigma)
= \frac{1}{h^2} (W_2^2(\mu_{t+h},\sigma)-2W_2^2(\mu_t,\sigma)+W_2^2(\mu_{t-h},\sigma))
\,,
\end{equation}
where $D_\tau$ denotes the symmetric difference operator with step $\tau>0$, that is, 
\begin{equation}
(D_\tau W_2^2)(\mu_t^X,\sigma)\coloneqq \frac{1}{2\tau}(W_2^2(\mu_{t+\tau}^X,\sigma)-W_2^2(\mu_{t-\tau}^X,\sigma)) \,.
\end{equation}
Recalling that $\gamma_t^\tau$ is a coupling plan between $\mu_{t+\tau}^X$ and $\sigma$ for any $\tau\in[-h,h]$, we obtain
\begin{equation}\label{eqn:W2_upper}
W_2^2(\mu_{t\pm\tau}^X,\sigma) \leq \int_{\Rb^{2d}} \|x-y\|^2 \rmd \gamma_t^{\pm\tau}(x,y) = \int_{\Rb^{3d}} \|\Phi_{\pm\tau}^X(x,v)-y\|^2 \rmd \mu_{t,x}(v) \rmd \gamma_t(x,y) \,.
\end{equation}
By making use of~\eqref{eqn:W2_upper}, for any $h\in(0,t)$, we have
\begin{equation}\label{eqn:deltaK_bound}
\begin{aligned}
\Delta_h\Kc(\mu_t^X,\sigma)
&\leq \frac{1}{h^2} \int_{\Rb^{3d}} \|\Phi_{\tau}^X(x,v)-y\|^2 -2\|x-y\|^2 +\|\Phi_{-\tau}^X(x,v)-y\|^2 \rmd \mu_{t,x}(v) \rmd \gamma_t(x,y) \\
&\leq \frac{1}{h} \int_{-h}^h \int_{\Rb^{2d}} \|\partial_\tau\Phi_\tau^X(x,v)\|^2 \rmd \mu_t(x,v) \rmd \tau\\
& \quad +\frac{1}{h^2} \int_0^h \int_{\Rb^{3d}} (h-\tau) \left\langle x-y, \partial_\tau^2\Phi_\tau^X(x,v) \right\rangle \rmd \mu_{t,x}(v) \rmd \gamma_t(x,y) \rmd \tau \\
&\quad +\frac{1}{h^2} \int_{-h}^0 \int_{\Rb^{3d}} (h+\tau) \left\langle x-y, \partial_\tau^2\Phi_\tau^X(x,v) \right\rangle \rmd \mu_{t,x}(v) \rmd \gamma_t(x,y) \rmd \tau\\
&= \int_{-1}^1 \int_{\Rb^{2d}} 
\|F(t+sh,\Phi_{sh}^V(x,v))\|^2
\rmd \mu_{t+sh} \rmd s \\
& \quad + \int_0^1 (1-s) \int_{\Rb^{3d}} \left\langle x-y, J(t+sh,\mu_{t+sh},\Phi_{sh}(x,v) \right\rangle \, \rmd \mu_{t,x}(v) \rmd \gamma_t(x,y) \rmd s \\
& \quad + \int_{-1}^0 (1+s) \int_{\Rb^{3d}} \left\langle x-y, J(t+sh,\mu_{t+sh},\Phi_{sh}(x,v) \right\rangle \, \rmd \mu_{t,x}(v) \rmd \gamma_t(x,y) \rmd s \,,
\end{aligned}
\end{equation}
where in the second inequality we use the fundamental theorem of calculus and Jensen's inequality, and in the last equality, we use~\eqref{eqn:pushforward_second}.

For a fixed $T>0$, we choose an integer $N>0$ such that $h=T/N$. Let $\{\mu_{nh}\}_{n=0}^N$ be recursively defined by $\mu_{(n+1)h}=(\Phi_h^n)_\sharp\mu_{nh}$ for $n=0,\dots,N$ where $\Phi_h^n=(\Phi_h^{n,X},\Phi_h^{n,V})$ satisfies
\begin{equation}
\begin{cases}
\partial_\tau \Phi^{n,X}_\tau(x,v) &= 
F(nh+\tau,\Phi^{n,V}_\tau(x,v)) \,,\\
\partial_\tau \Phi^{n,V}_\tau(x,v) &= 
-G(nh+\tau,\mu_{nh+\tau},\Phi^{n,X}_\tau(x,v)) \,,
\end{cases} \quad \Phi^n_0(x,v)=(x,v) \quad \text{ for } \mu_{nh}\text{-a.e.}\,(x,v)\,,
\end{equation}
with initial condition $(\Phi_0^{n,X},\Phi_0^{n,V})=(x,v)$ and $\tau\in(-h,h)$. Then, for $n=0,\dots,N$,~\eqref{eqn:deltaK_bound} provides the inequality
\begin{equation}\label{eqn:deltahK_nh}
\begin{aligned}
\Delta_h\Kc(\mu_{nh}^X,\sigma)
\leq& \int_{-1}^1 \int_{\Rb^{2d}} 
\left\|F((n+s)h,\Phi_{sh}^V(x,v))\right\|^2
\rmd \mu_{(n+s)h} \rmd s \\
& + \int_0^1 (1-s) \int_{\Rb^{3d}} \left\langle x-y, J((n+s)h,\mu_{(n+s)h},\Phi_{sh}(x,v) \right\rangle \, \rmd \mu_{nh,x}(v) \rmd \gamma_{nh}(x,y) \rmd s \\
& + \int_{-1}^0 (1+s) \int_{\Rb^{3d}} \left\langle x-y, J((n+s)h,\mu_{(n+s)h},\Phi_{sh}(x,v) \right\rangle \, \rmd \mu_{nh,x}(v) \rmd \gamma_{nh}(x,y) \rmd s \\
&\eqqcolon (A) + (B) + (C) \,.
\end{aligned}
\end{equation}
Multiplying the inequality with $h$ and summing over $n=1,\dots,N-1$ yields for the LHS
\begin{equation}
\begin{aligned}
\sum_{n=1}^{N-1} h \Delta_h \Kc(\mu_{nh},\sigma) &= (D_{h/2}W_2^2)(\mu_{(N-1/2)h},\sigma) - (D_{h/2}W_2^2)(\mu_{h/2},\sigma) \\
&= \frac{1}{h} \left(\int_{(N-1)h}^{Nh} - \int_{0}^{h} \right)\frac{\rmd}{\rmd \tau} W_2^2(\mu_\tau^X,\sigma) \rmd \tau \\
&= 2 \int_{0}^{1}\int_{\Rb^{3d}}\left\langle F((N-1+s)h,v), x-y \right\rangle \, \rmd \mu_{(N-1+s)h,x}\rmd\gamma_{{(N-1+s)h}}(x,y) \rmd s \\
& - 2 \int_{0}^{1}\int_{\Rb^{3d}}\left\langle F(sh,v), x-y \right\rangle \rmd \mu_{sh,x} \rmd\gamma_{{sh}}(x,y) \rmd s \,.
\end{aligned}
\end{equation}
Passing to the limit $h\to0$ with $Nh=T$ gives
\begin{equation}
\begin{aligned}
\lim_{h\to0} \sum_{n=1}^{N-1} h \Delta_h \Kc(\mu_{nh},\sigma)
&= 2\int_{\Rb^{3d}} \left\langle F(T,v), x-y \right\rangle \, \rmd \mu_{T,x} \rmd\gamma_{T}(x,y) \\
& \quad - 2\int_{\Rb^{3d}} \left\langle F(0,v), x-y \right\rangle \rmd \mu_{0,x} \rmd\gamma_{0}(x,y) \\
&= \frac{\rmd}{\rmd t} W_2^2(\mu_T^X,\sigma) - \frac{\rmd}{\rmd t} W_2^2(\mu_0^X,\sigma) \,,
\end{aligned}
\end{equation}
which holds due to the dominated convergence theorem. On the other hand, the following convergences hold for the terms on the RHS of~\eqref{eqn:deltahK_nh}:
\begin{equation}
\begin{aligned}
&\sum_{n=1}^{N-1} hA\;\; \longrightarrow\;\; 2 \int_0^T \int_{\Rb^{2d}} \|F(t,v)\|^2 \rmd \mu_t \rmd t \\
&\sum_{n=1}^{N-1} hB\;\; \longrightarrow\;\; \frac{1}{2} \int_0^T \int_{\Rb^{3d}} \left\langle x-y, J(t,\mu_t,x,v)\right\rangle \, \rmd \mu_{t,x}(v) \rmd \gamma_t(x,y) \rmd t \\
&\sum_{n=1}^{N-1} hC\;\; \longrightarrow\;\; \frac{3}{2} \int_0^T \int_{\Rb^{3d}} \left\langle x-y, J(t,\mu_t,x,v)\right\rangle \, \rmd \mu_{t,x}(v) \rmd \gamma_t(x,y) \rmd t 
\end{aligned}
\end{equation}
by the definition of Riemann integrable functions and the assumed regularity~\eqref{eqn:assumption_regularity}.
\end{proof}

\section{Local convexity of neural network training}\label{app:nn_training}

For the neural network architecture proposed above, the loss functional~\eqref{eqn:energy_NN} is not geodesically convex over $\mathcal{P}(\mathbb{R}^d)$, but we claim it is locally convex along geodesics satisfying certain conditions (cf.\ \eqref{eqn:geo_assumption} below). We discuss the argument explicitly in this section.

Consider the training of an infinitely wide $2$-layer neural network with the loss functional
\[
    E[\rho] = \frac{1}{2}\iint_{\Rb\!\times\Rb^d} |y-g(x,\rho)|^2\, \rmd \pi(x,y)\,,
\]
where $\pi$ is a given distribution over the sampled data and the function $g$ is a two-layer neural network defined according to:
\[
    g(x,\rho) \coloneqq \int_{\Rb\times\Rb^d} V(x,z)\,\rmd \rho(z),\quad\text{with}\quad V(x,(\alpha,w)) = \alpha\,\sigma(w\cdot x)\,,
\]
for all $(x,\rho)\in\Rb^d\times \Pc(\Rb\!\times \Rb^d)$. Here $x$ is the input to the neural network, $\rho$ is the probability measure according to which the neuron weights are drawn, and $\sigma$ is the positively $1$-homogeneous ReLU function. We slightly modify the representation of $V$. Noting that for $\alpha\in\Rb$, $\alpha = \alpha 1_{\{\alpha>0\}}-|\alpha| 1_{\{\alpha<0\}}$, so we rewrite:
\[
\begin{aligned}
    V(x,(\alpha,w)) &= \sigma(\alpha 1_{\{\alpha>0\}} w\cdot x) - \sigma(|\alpha| 1_{\{\alpha<0\}} w\cdot x) \\
    &\eqqcolon \sigma(\omega_1\cdot x) - \sigma(\omega_2\cdot x) \eqqcolon \widehat V(x,\omega)\,,
    \end{aligned}
\]
where we defined $\omega_1=\alpha 1_{\{\alpha>0\}} w$ and $\omega_2=|\alpha| 1_{\{\alpha<0\}} w$. This relation forms the definition:
\[
\Rb^d\times\Rb^d\ni\omega=\sfr(\alpha,w)=(\omega_1,\omega_2)\,,
\]
and correspondingly, we define: $\widehat\rho = \sfr_\#\rho \in \Pc(\Rb^d\times\Rb^d)$. 
We thus obtain 
\begin{align*}
    g(x,\rho) &= \iint_{\Rb\times \Rb^d} V(x,(\alpha,w))\,\rmd\rho(\alpha, w) = \iint_{\Rb\times\Rb^d} \bigl[\sigma(\sfr_1(\alpha,w)\cdot x)-\sigma(\sfr_2(\alpha,w)\cdot x)\bigr]\,\rmd \rho(\alpha, w) \\
    &= \int_{\Rb^{2d}} \bigl[ \sigma(\omega_1\cdot x) - \sigma(\omega_2\cdot x)\bigr]\,\rmd\widehat\rho( \omega) = \int_{\Rb^{2d}} \widehat V(x,\omega)\,\rmd\widehat\rho(\omega) \eqqcolon \widehat g(x,\widehat\rho)\,.
\end{align*}
Consequently, we relax the training of $\rho$ using $E[\rho]$ into a different problem: training $\hat{\rho}$ using the following functional:
\[
    E[\rho] = \frac{1}{2}\iint_{\Rb\times\Rb^d} |y-\widehat g(x,\widehat \rho)|^2\,\rmd\pi(x,y) \eqqcolon \widehat E[\widehat \rho]\,.
\]

We now show that $\widehat E$ is \emph{locally} geodesically convex when $\sigma$ is the ReLU function. To show geodesically convexity at a probability measure $\hat{\rho}$ amounts to show the objective functional is convex along any geodesics whose origin is at $\hat{\rho}$. To do so, we set $\hat{\eta}$ to be any probability measure in $\mathcal{P}(\Rb^d\times\Rb^d)$, and denote by $\mathsf{T}$ the optimal transport map between $\widehat\rho$ and $\widehat\eta$. 
Then, along the geodesics with a constant speed~\citep{Mc:1997convexity}, $\widehat\gamma_t\coloneqq[(1-t)id + t\mathsf{T}]_\#\widehat\rho$, for $t\in[0,1]$, we have
\begin{align*}
    \widehat g(x,\widehat\gamma_t) &= \int_{\Rb^{2d}} \widehat V(x,\omega)\,\rmd\widehat\gamma_t(\omega) = \int_{\Rb^{2d}} \widehat V(x,(1-t)\omega+t\mathsf{T}(\omega))\,\rmd\widehat\rho(\omega)\,.
\end{align*}
If we can successfully rewrite as
\begin{align}
    \widehat g(x,\widehat\gamma_t) &= (1-t) \int_{\Rb^{2d}} \widehat{V}(x,\omega )\,\rmd\rho(\omega) + t \int_{\Rb^{2d}} \widehat V(x,\mathsf{T}(\omega))\,\rmd\widehat\rho(\omega) \label{eqn:geo_assumption}\\
    &= (1-t) \int_{\Rb^{2d}} \widehat{V}(x,\omega )\,\rmd\rho(\omega) + t \int_{\Rb^{2d}} \widehat V(x,\omega')\,\rmd\eta(\omega') \nonumber\\
    &= (1-t)\widehat g(x,\widehat\rho) + t\widehat g(x,\widehat\eta),\nonumber
\end{align}
we are showing $\widehat\rho\mapsto \widehat g(x,\widehat\rho)$ is geodesically linear for every $x\in\Rb^d$. Owing to the convexity of $r\mapsto |r|^2$, we obtain
\begin{align*}
    \widehat E[\widehat\gamma_t] &= \frac{1}{2}\iint_{\Rb\times\Rb^d} |y-\widehat g(x,\widehat\gamma_t )|^2\,\rmd\pi(x,y) = \frac{1}{2}\iint_{\Rb\times\Rb^d} |(1-t)(y-\widehat g(x,\widehat\rho)) + t(y-\widehat g(x,\widehat\eta))|^2\,\rmd\pi(x,y) \\
    &\le (1-t)\frac{1}{2} \iint_{\Rb\times\Rb^d} |y-\widehat g(x,\widehat\rho)|^2 \,\rmd\pi(x,y) + t\frac{1}{2} \iint_{\Rb\times\Rb^d} |y-\widehat g(x,\widehat\eta)|^2 \,\rmd\pi(x,y) = (1-t)\widehat E[\widehat\rho] + t \widehat E[\widehat\eta],
\end{align*}
thus implying the (local) geodesic convexity of $\widehat E$ when~\eqref{eqn:geo_assumption} holds.

Generally speaking, however, \eqref{eqn:geo_assumption} does not hold for any given $\widehat{\eta}$.
It {\em would} hold if the optimal transport map $\mathsf{T}$ between $\widehat\rho$ and $\widehat\eta$ were to satisfy
\[
    \text{sign}(\omega_i\cdot x) = \text{sign}(\mathsf{T}_i(\omega)\cdot x), \quad\text{for $\widehat\rho$-almost every $\omega$}\,,
\]
with the convention that $\text{sign}(0)=0$. When this happens, we would have
\[
    \text{sign}(((1-t)\omega_i\cdot x + t\mathsf{T}_i(\omega)\cdot x) = \text{sign}(\omega_i\cdot x) = \text{sign}(\mathsf{T}_i(\omega)\cdot x),\quad\text{for every $t\in[0,1]$}\,,
\]
and noticing $\sigma(\alpha + \beta) = \sigma(\alpha) + \sigma(\beta)$ for $\alpha\,\beta\ge 0$, we can split
\begin{align*}
\widehat V(x,\omega) = \widehat V(x,(1-t)\omega+t\mathsf{T}(\omega))= (1-t) \widehat{V}(x,\omega ) + t \widehat V(x,\mathsf{T}(\omega))\,,
\end{align*}
ensuring~\eqref{eqn:geo_assumption}.

\vskip 0.2in
\bibliography{ref_jmlr}

\end{document}